\documentclass[final,hidelinks]{siamltex704}
\usepackage{amssymb,amsmath,graphicx,amscd,mathrsfs}

\makeatletter

\newcommand{\Rmnum}[1]{\expandafter\@slowromancap\romannumeral #1@}
\makeatother

\usepackage{psfrag,graphicx,amsmath,amssymb,subfigure,multirow,graphpap,cite,hyperref,color}
\usepackage{multirow}

\newtheorem{remark}{Remark}[section]

\numberwithin{equation}{section} 
\numberwithin{figure}{section}
\numberwithin{table}{section}

\def\a{\alpha}

\def\G{\Gamma}

\renewcommand\o{\omega}\renewcommand\O{\Omega}

\def\be{{\bf e}}

\def\bn{{\bf n}}
\def\bt{{\bf t}}

\def\sN{{_N}}

\def\cE{{\mathcal E}}

\def\cN{{\mathcal N}}

\def\cT{{\mathcal T}}

\newcommand{\Ei}{{{\cE}^{int}}}

\newcommand{\jump}[1]{\ensuremath{[\![#1]\!]} }

\def\f12{\frac12}
\def\dfrac{\displaystyle\frac}

\newcommand{\gperp}{\nabla^{\perp}}

\newcommand{\bdm}{\begin{displaymath}}
\newcommand{\edm}{\end{displaymath}}
\newcommand{\beq}{\begin{equation}}
\newcommand{\eeq}{\end{equation}}
\newcommand{\beqa}{\begin{eqnarray}}
\newcommand{\eeqa}{\end{eqnarray}}
\newcommand{\beqas}{\begin{eqnarray*}}
\newcommand{\eeqas}{\end{eqnarray*}}

\begin{document}
\title{Residual-based a posteriori error estimation for Immersed Finite Element Methods}

\author{
Cuiyu He\thanks{Department of Mathematics, University College London, London, UK, WC1H 0AY ({\tt c.he@ucl.ac.uk})}
\and 
Xu Zhang\thanks{Department of Mathematics, Oklahoma State University, Stillwater, OK 74078, ({\tt xzhang@okstate.edu}). The research of this author was partially supported by the National Science Foundation grant DMS-1720425.}}

\date{}
\maketitle

\begin{abstract}
In this paper we introduce and analyze the residual-based a posteriori error estimation of the partially penalized immersed finite element method for solving elliptic interface problems. The immersed finite element method can be naturally utilized on interface-unfitted meshes. Our a posteriori error estimate is proved to be both reliable and efficient with reliability constant independent of the location of the interface. Numerical results indicate that the efficiency constant is independent of the interface location and that the error estimation is robust with respect to the coefficient contrast. 
\end{abstract}



\section{Introduction}
Interface problems arise widely in the multi-physics and multi-material applications in the fluid mechanics and material science. The governing partial differential equations (PDEs) for interface problems are usually characterized with discontinuous coefficients that represent different material properties. The solutions to the interface problems often involve kinks, singularities, discontinuity, and other non-smooth behaviors. It is therefore challenging to obtain accurate numerical approximations for interface problems. Moreover, the complexity of the interface geometry may add an extra layer of difficulty to the numerical approximation. 

In general, there are two classes of numerical methods for solving interface problems. The first class of methods use interface-fitted meshes, i.e., the meshes are tailored to fit the interface, see the left plot in Figure \ref{fig: domain}. Methods of this type include classical finite element methods \cite{ChZou:98}, discontinuous Galerkin methods \cite{ArBrCoMa:02} and recently developed weak Galerkin methods \cite{2016MuWangYeZhao}, to name a few.  The second class of methods use unfitted meshes which are independent of the interface, as illustrated in the right plot in Figure \ref{fig: domain}. In the past few decades, many numerical methods based on unfitted meshes have been developed. In the finite difference framework, since the pioneering work of immersed boundary method \cite{2002Peskin} by Peskin, many numerical methods of finite difference type have been developed such as immersed interface method \cite{LeLi:94, 2006LiIto}, matched interface and boundary method \cite{2004ZhaoWei}. In finite element framework, there are quite a few numerical methods developed, for instance, the general finite element method  \cite{BaOs:83}, unfitted finite element method \cite{HaHa:02}, multi-scale finite element method  \cite{HouWuCai:99}, extended finite element method \cite{1999MoesDolbowBelytschko}, and immersed finite element method (IFEM) \cite{Li:98, 2003LiLinWu}. A great advantage for unfitted numerical methods is that they can circumvent (re)meshing procedure which can be very expensive especially for time dependent problems with complex interface geometry or for shape optimization processes that require repeated updates of the mesh. 

\begin{figure}[ht]
\begin{center}
\includegraphics[width=0.4\textwidth]{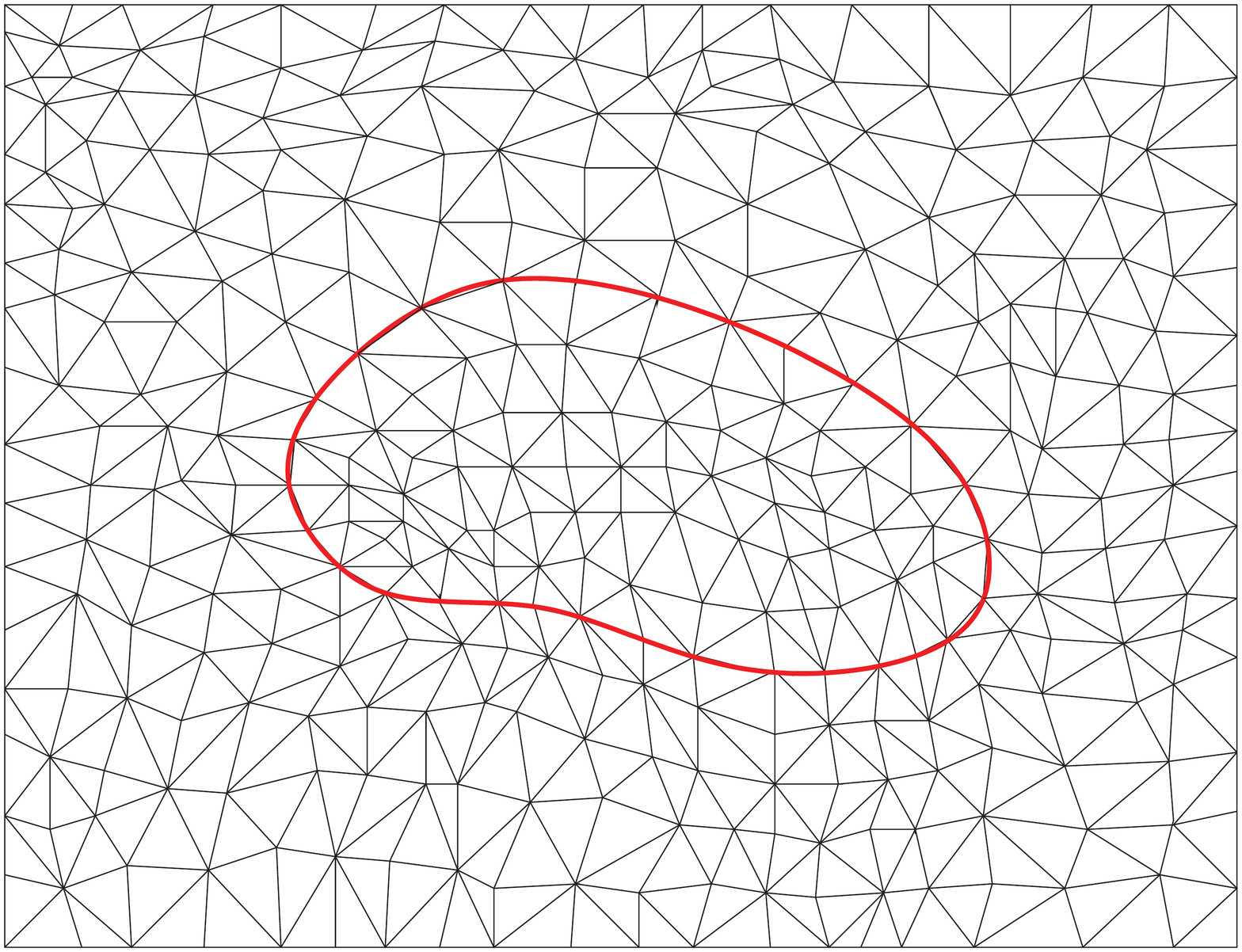}~~
\includegraphics[width=0.4\textwidth]{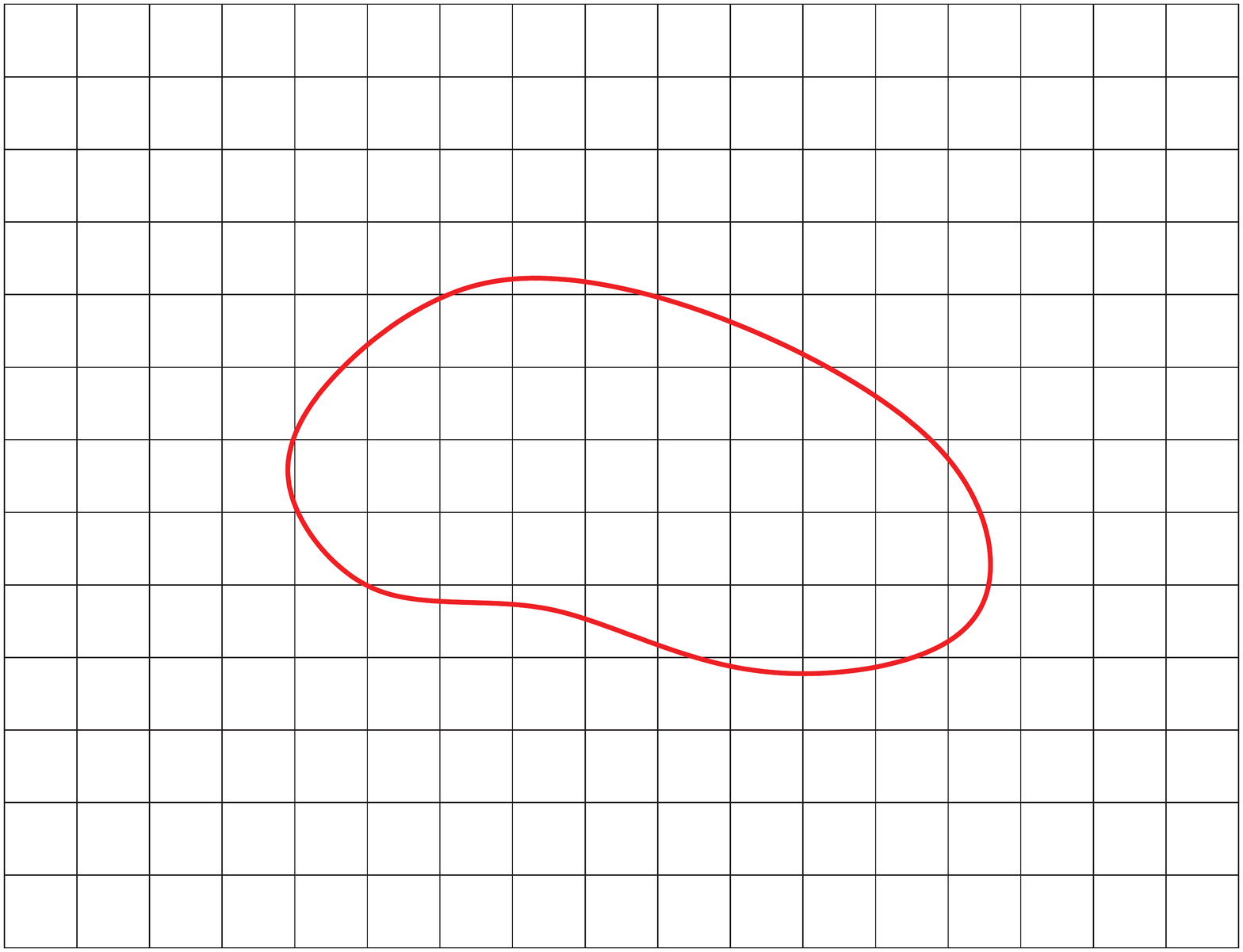}
\end{center}
\caption{An interface-fitted mesh (left), and an unfitted mesh (right).}
\label{fig: domain}
\end{figure}

The IFEM was first developed in \cite{Li:98} for a one-dimensional elliptic interface problem
and then extended to higher-order approximations \cite{AdjeridGuoLin:2017, AdjeridLin:2009, CaZhZh:17, 2017CaoZhangZhangZou} and to higher-dimensional elliptic interface problems \cite{GuSaSa:16, 2011HeLinLin, LiLinLinRo:04, 2018LinSheenZhang, 2010VallaghePapadopoulo, 2011WuLiLai}.
Recently, the partially penalized immersed finite element method was introduced in \cite{LinLinZhang:15}. 
Compared to classical IFEM, the partially penalized IFEM contains normal flux jump terms on interface edges to ensure the consistency of the scheme. In addition, the new IFEM includes a stabilization term on interface edges to guarantee the stability of the scheme. The partially penalized IFEM significantly improves the numerical approximation, especially the accuracy around the interface. The optimal a priori error estimate is theoretically proved for partially penalized IFEM in the energy norm \cite{LinLinZhang:15}. 

We note that the partially penalized IFEM can solve elliptic interface problems accurately on uniform Cartesian meshes provided that the exact solution of interface problems is piecewise smooth, and the contrast of the coefficient is ``moderate". However, for interface problems that also involve singularity or steep gradient,
the partially penalized IFEM alone may not be efficient to obtain an accurate approximation on uniform meshes. In such cases, it is necessary to apply certain adaptive mesh refinement (AMR) strategy to IFEM. The goal for AMR is to obtain an approximate solution within prescribed error tolerance with minimum computational cost which is particularly rewarding for interface problems with non-smooth solutions. 

The key success for AMR is the a posteriori error estimation which provides both global and local information on the approximation error. Moreover, in many applications even without the intention of performing adaptive mesh refinement, the a posteriori error estimation is also important in assessing the quality of the simulation by providing an effective error control.
We note that the a priori error estimation of IFEM is getting mature in the past decade, but the a posteriori error estimation of IFEM is still in the infancy. In this paper, we will develop and analyze a residual-based a posteriori error estimate for the partially penalized IFEM for second-order elliptic interface problems in the two dimensional space.

Comparing to the residual-based error estimation for the classical finite element solution \cite{BeVe:00}, the newly introduced estimator in this work additionally
includes the jump of tangential derivative of numerical solutions on interface edges
besides the standard terms of element residual and jump of normal flux on all edges. This is necessary since the numerical solution is in general discontinuous across the interface edges due to the
construction of the IFEM basis functions. 
Moreover, the new error estimator also includes the geometrical fitting error due to the polygonal approximation of the curved interface.   
Theoretically, we prove that the error estimation is both globally reliable and locally efficient. 
To prove reliability, we use the Helmholtz decomposition and  a $L^2$ representation technique recently introduced in \cite{CaHeZh-mcom:17, CaHeZh-sinum:17}; moreover, we introduce a new type of Cl\'ement-type interpolation in the IFEM space that allows us to take advantage of the error equation of IFEM. 
In the efficiency analysis for IFEM, the technique using the standard bubble functions \cite{Ve:1996} is invalid because the edge jumps of the normal flux and the tangential derivative become piecewise constant on the interface edges. 
\textcolor{black}{Instead, we prove the efficiency in two different approaches that aims to provide an optimal efficiency constant for both regular and irregular interface edges.}

The rest of the article is organized as follows. In Section \ref{sec:2}, we recall the partially penalized IFEM for elliptic interface problems. In Section \ref{sec:3}, we introduce our residual-based error estimator specially designed for IFEM. Section \ref{sec:4}  and  Section \ref{sec:5} are dedicated to the analysis of global reliability and local efficiency, respectively. 
Finally, in  Section \ref{sec:6}, we present several numerical experiments to test the performance of our a posteriori error estimators.

\section{Interface Problems and Partially Penalized IFEM}\label{sec:2}
Let $\O\subset \mathbb{R}^2 $ be a polygonal domain with Lipschitz boundary
$\partial \O = \overline \Gamma_D \cup \overline \Gamma_N$, where 
$ \Gamma_D \cap  \Gamma_N = \emptyset$. Assume that $\mbox{meas}(\Gamma_D) > 0$.
We consider the  elliptic interface problem:
\beq  \label{problem}
		-\nabla \cdot (\a \nabla u) =f \quad \mbox{ in } {\O^+\cup\O^-}
\eeq
with boundary conditions
\[
u= 0 ~\mbox{ on } \Gamma_D~~
\mbox{ and } ~~
-\a \nabla u \cdot \bn=g_\sN ~\mbox{ on } \Gamma_N.
\]
Here,
$f \in L^2(\O)$, $g_\sN \in L^2(\Gamma_N)$, and $\bn$ is the unit vector outward normal to $\partial \O$. The notations $\nabla$ and $\nabla \cdot$ are the gradient and divergence operators, respectively.  Furthermore, assume that $\O$ is separated by a closed smooth interface curve $\Gamma$ into $\O^+$ and $\O^-$ such that 
$\overline{\O} = \overline{\O^+\cup\G\cup\O^-}$. The diffusion coefficient $\a$ is assumed to be a positive piecewise constant function as follows
\[
\a(x,y)=\left\{ \begin{array}{lll}
	\a^+ & \mbox{for} \,(x, y) \in \O^+,\\[2mm]
	\a^-  & \mbox{for} \,(x,y) \in \O^-.
\end{array}
\right.
\]
Denote by $\rho = \dfrac{\a^+}{\a^-}$ the ratio of the coefficient jump.
The solution is assumed to satisfy the following interface jump conditions:
\beq
	\jump{u}_{\G} = 0 \quad \mbox{and} \quad \jump{\alpha \nabla u\cdot \bn}_{\G} = 0,
\eeq
where the jump of a function $v$ across the interface $\G$ is defined by
\[
\jump{v}_{\G} = v^+|_\G - v^-|_\G.
\]

We use the standard notations for the Sobolev spaces.
Let
\[
	H_D^1(\O)=\{v \in H^1(\O) : v =0 \mbox{ on } \Gamma_D\}.
\]
Then the variational problem for \eqref{problem} is to find $u \in H_D^1(\O)$ such that
\beq\label{continuous-weak-solution}
	a(u,v) \triangleq (\a \nabla u, \nabla v) =(f,v) -(g_\sN,v)_{\Gamma_N}, \quad \forall \, v\in H_D^1(\O),
\eeq
where $(\cdot,\cdot)_\o$ is the $L^2$ inner product on $\o$. The subscript $\o$ is omitted when $\o=\O$.

\subsection{Triangulation}
In this paper, we only consider the triangular meshes in two dimensions. 
Let $\cT =\{K\}$ be a triangulation of $\O$ that is regular but not necessarily body-fitted. 
Denote the set of all vertices of the triangulation $\cT$ by 
\[\cN:= \cN_I \cup \cN_D \cup \cN_N\]
where $\cN_I$ is the set of all interior vertices, and $\cN_D$ and $\cN_N$ are the sets of vertices
on $\bar\Gamma_D$ and $\Gamma_N$, respectively. 
Denote the set of all edges of the triangulation $\cT$ by
\[
\cE := \cE_I \cup \cE_D \cup \cE_N
\]
where $\cE_I$ is the set of all interior edges and $\cE_D$ and $\cE_N$ are
the sets of boundary edges on $\Gamma_D$ and $\Gamma_N$, respectively. 
For each element $K \in \cT$, denote by $h_K$ the diameter of $K$,
and by $\cN_K$ and $\cE_K$ the sets of 
all vertices and edges on $K$, respectively.

For simplicity, we assume that the interface cuts the partition
with the following properties:
\begin{description}
\item[(I)] If $\Gamma$ meets an edge at more than one point, then this edge is part of $\Gamma$.
\item[(II)] If the case (I) does not occur,  then $\Gamma$ must intersect a triangle at two points, and these two points must be on different edges of this triangle.
\end{description}
Based on the above assumption, all triangular elements in the partition can be categorized into two classes:  \textit{non-interface elements} that either has no intersection with $\Gamma$ or $\Gamma \cap  K \subset \partial K$, and \textit{interface elements} whose interior is cut through by $\Gamma$.

Denote the set of all interface elements by  $\cT^{int}$.
For each interface triangle $K$ we let $\Gamma_K = \Gamma \cap K$ and 
$\tilde\Gamma_K$ be the line segment approximating $\Gamma_K$ by connecting two endpoints of $\Gamma_K$.
Let $K^+ = K \cap \O^+$ and $K^- = K \cap \O^-$. Also we let $\tilde K^+$ and $\tilde K^-$ be the two sub-elements of $K$ separated by $\tilde \Gamma_K$.
From the setting above, it is easy to see that $\tilde K^\pm$ is either a triangle or a quadrangle. Also we define 
\begin{equation}\label{eq: SK}
S_K \triangleq(K^+ \setminus \tilde K^+) \cup (K^- \setminus \tilde K^-), 
\quad \forall \,K \in \cT^{int},
\end{equation}
which is the region enclosed by $\Gamma_K$ and $\tilde\Gamma_K$. Under the assumption that $\Gamma_K$ is $C^2$-continuous, the area of the $S_K$ is of at least $O(h^3_K)$ \cite{LiLinLinRo:04}. Figure \ref{fig: triangle} provides an illustration of a typical triangular interface element.

\begin{figure}[hbt!]
\begin{center}
\includegraphics[width=.35\textwidth]{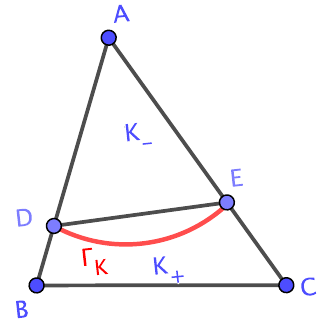}
\end{center}
\caption{A triangular interface element}
\label{fig: triangle}
\end{figure} 

For an edge $F \in \cE$, if $F$ is cut through by $\Gamma$, i.e., $F \cap \Gamma \neq \emptyset$ and
$F \not \subset \Gamma$, then $F$ is called an interface edge and
denote by $\Ei$ the set of all such interface edges.
For each $F \in \cE$, denote by $h_F$ the length of $F$.
Denote by $\bn_F =(n_1, n_2)$ and $\bt_F=(-n_2, n_1)$ the unit vectors
normal and tangential to $F$, respectively. 
 Let $K_{F,1}$ and $K_{F,2}$ be the two elements sharing the common edge $F \in \cE_I$
such that the unit  vector out normal to $K_{F,1}$ coincides with $\bn_F$. When $F\subset \partial \O$, 
$\bn_F$ is the unit outward vector normal to $\partial \O$, and denote by $K_{F,1}$ the boundary element having the
edge $F$. 
For a function $v$ that is defined on $K_{F,1} \cup K_{F,2}$, denote its traces on $F$ by $v|_F^1$ and $v|_F^2$
restricted on $K_{F,1}$ and $K_{F,2}$, respectively.
 Define the jump of a function $v$ on the edge $F$ by
\[
	\jump{v}_F = \left\{
	\begin{array}{lll}
		v|_F^1 -v|_F^2,& \mbox{for }F \in \cE_I,\\[2mm]
		v|_F^1,& \mbox{for } F \in \cE_D \cup \cE_N
	\end{array}
	\right.
\]
and the average of a function $v$ on the edge $F$ by
\[
		\{v\}_F = \left\{
	\begin{array}{lll}
		\left( v|_F^1 +v|_F^2 \right)/2, & \mbox{for }F \in \cE_I,\\[2mm]
		 \,v|_F^1, & \mbox{for } F \in \cE_D \cup \cE_N.
	\end{array}
	\right.
\]
It is easy to verify that 
\beq\label{jump-formula}
	\jump{vw}_F= \jump{v}_F \{w\}_F+ \{v\}_F \jump{w}_F, \quad \forall \, F \in \cE.
\eeq
For simplicity, we may drop the subscript $F$ in the notations $\jump{\cdot}_F$ and $\{\cdot\}_F$ if there is no confusion on where the jump and average are defined.

\subsection{IFEM approximation}
For simplicity, we assume that the interface does not intersect with the boundary, i.e., $\Ei \subset \cE_I$. Let $\tilde \a$ be an approximation of $\a$ such that
\[
	\tilde \a(x,y) = \left\{\begin{array}{lll}
		\a^+ & \mbox{if}& (x,y) \in \tilde K^+,\\[2mm]
		\a^- & \mbox{if}& (x,y) \in \tilde K^-,
	\end{array}
	\right.
	\quad \forall \,(x,y) \in K \in \cT^{int}.
\]
For each interface element $K \in \cT^{int}$, define the local IFE space by 
\[
	\tilde P_1(K) = \left\{ v \in H^1(K):  \tilde \a \nabla v \in H(\mbox{div},K), v|_{\tilde K^\pm} \in P_1(\tilde K^\pm)\right\}
\]
where $P_1(w)$ is the space of all polynomial functions in $w$ of degree no more than $1$.
The global IFE space $\mathcal{S}(\cT)$ is then defined to include all functions such that 
\begin{enumerate}
\item $v|_K \in \tilde P_1(K)$ for all $K \in \cT^{int}$ and $v|_K \in P_1(K)$ for all $K \in \cT/\cT^{int}$, and 
\item $v$ is continuous at every vertex $z \in \cN$.
\end{enumerate}
Note that for each $z \in \cN$, there exists a unique IFE nodal basis function \cite{LiLinLinRo:04,2003LiLinWu}, denoted by $\tilde\lambda_z \in \mathcal{S}(\cT)$, such that
\[
	\tilde\lambda_z(z') = \delta_{zz'}, \quad \forall \, z' \in \cN
\]
where $\delta$ is the Kronecker delta function.

The partially penalized IFEM solution for the interface problem is to find \\$u_\cT \in \mathcal{S}_D(\cT) = \left\{ v \in \mathcal{S}(\cT) \,:\, v =0 \mbox{ on } \Gamma_D \right\}$ such that 
\beq \label{IFE approx}	
	a_h(u_\cT,v)=(f,v) - (g_\sN, v)_{\Gamma_N}, \quad \forall \,v \in \mathcal{S}_D(\cT)
\eeq
where the bilinear form $a_h(w,v )$ is defined by
\begin{eqnarray*} \label{FE approximation}
	a_h(w,v)
	&=&\sum_{K\in \cT} \int_{K} \tilde\a \nabla w  \cdot \nabla v \,dx
	-\sum_{F \in  \Ei} \int_F\{\tilde \a \nabla w \cdot \bn_F \} \jump{v} \,ds\\[2mm]
	&&+ \epsilon \sum_{F\in \Ei} \int_F \{ \tilde \a \nabla v \cdot \bn_F\} \jump{w} \,ds
	+\sum_{F\in \Ei} \gamma\int_F \dfrac{\tilde \a}{h_F} \jump{w}\jump{v} \,ds. 
\end{eqnarray*}
Here $\epsilon$ may take the values $-1$, $0$, and $1$, corresponding to symmetric, incomplete, and non-symmetric IFEM. The constant $\gamma$ is the
stability parameter and needs to be chosen large enough for symmetric and incomplete IFEMs guarantee the coercivity. For non-symmetric IFEM, the constant $\gamma$ is only required to be positive. For more details on the partially penalized IFEM, we refer readers to \cite{LinLinZhang:15}. 

\begin{remark}
By the definition of $\tilde \a$ and $\tilde \Gamma_K$ it is easy to see that
$\tilde \a = \a $ on all $F \in \cE^{int}$.
\end{remark}

\subsection{Inconsistency error} 
Due to the geometrical approximation of the interface curve $\Gamma$ by a polygonal interface $\tilde \Gamma = \bigcup\limits_{K\in\cT^{int}} \tilde \Gamma_K$, the following geometrical inconsistent error exists. By \eqref{IFE approx} and integration by parts we have for any $v \in \mathcal{S}_D(\cT)$
\begin{equation} \label{eq: inconsistency}
\begin{split}
	a_h(u - u_\cT,v)
	=&\sum_{K\in \cT} \int_{K} \tilde\a \nabla u \cdot \nabla v \,dx
	-\sum_{F \in  \Ei} \int_F\{\tilde \a \nabla u  \cdot \bn_F \} \jump{v} \,ds - (f,v) + (g_\sN, v)_{\Gamma_N} \\
	=&\sum_{K\in \cT} \int_{K} \a \nabla u \cdot \nabla v \,dx 
	-\sum_{F \in  \Ei} \int_F\{\a \nabla u  \cdot \bn_F \} \jump{v} \,ds \\
	&+ \sum_{K\in \cT} \int_{K} (\tilde\a - \a) \nabla u \cdot \nabla v \,dx - (f,v) +(g_\sN, v)_{\Gamma_N}
	\\
	=&
	 \sum_{K\in \cT^{int}} \int_{K} (\tilde\a - \a) \nabla u \cdot \nabla v \,dx. 
\end{split}
\end{equation}
\begin{remark}
If the interface $\Gamma$ is a polygon such that $\tilde\alpha = \alpha$ on each interface element, then the term \eqref{eq: inconsistency} vanishes. In this case, the partially penalized IFEM scheme is consistent. In general, for a curved interface $\Gamma$, the global convergence of IFEM will not be affected by such linear approximation of the interface, since the partially penalized IFEM uses piecewise linear approximation \cite{2007Braess}. 
\end{remark}

\section{Residual-Based A Posteriori Error Estimation}\label{sec:3}
In this section, we introduce the residual-based 
error estimator for the partially penalized IFEM.  
We note that the classical residual-based a posteriori error estimation for conforming finite element methods on fitted meshes consists of element residual and the jump of the normal flux on edges. For the IFEM, it is also necessary to include the jump of the tangential 
derivative on interface edges since the IFEM solution
may not be continuous across interface edges.

Define the normal flux jump of $u_\cT$ on each edge by
\[
	j_{n,F}=\left\{
	\begin{array}{lll}
		\jump{\a \nabla u_\cT \cdot \bn_F}_F, & \mbox{for }F \in \cE_I,\\[2mm]
		0,& \mbox{for }F \in \cE_D, \\[2mm]
		\a \nabla u_\cT \cdot \bn_F+g_\sN|_F, & \mbox{for }F \in \cE_N,
	\end{array}
	\right.
\]
and the tangential derivative jump of $u_\cT$ on each edge by 
\[
	j_{t,F}=\left \{
	\begin{array}{lll}
		\jump{ \nabla u_\cT \cdot \bt_F}_F, &\mbox{for } F \in \cE_I,\\[2mm]
		0,  & \mbox{for }F \in \cE_D \cup \cE_N.
	\end{array}
	\right.
\]
Note that on each interface edge $F \in \Ei$ both $j_{n,F}$ and $j_{t,F}$ are piecewise constant.

For all $K \in \cT$ we define the local error indicator $\eta_K$ by
\begin{equation} \label{eta-K}
\begin{split}
	\eta_K^2&=
		\sum \limits_{F \in \cE_K \cap \Ei} \left(\dfrac{h_F}{2} \|\tilde\a_F^{-1/2} j_{n,F}\|_{0,F}^2
			+\dfrac{h_F}{2}  \| \tilde\a_F^{1/2} j_{t,F}\|_{0,F}^2\right) 
			+ \|\tilde \a^{1/2} \nabla u_\cT \|_{S_K}^2
			\\
					&+
			\sum \limits_{F \in \cE_K \cap \cE_I  \setminus  \Ei }
			  \dfrac{h_F}{2} \| \tilde\a_F^{-1/2}j_{n,F}\|_{0,F}^2
			  +
		\sum \limits_{F \in \cE_K \cap \cE_N}  h_F \| \tilde\a_F^{-1/2}j_{n,F}\|_{0,F}^2
\end{split}
\end{equation}
where $\tilde\a_F(x)=\max\left(\tilde\a|_F^1(x), \tilde\a|_F^2(x)\right)$, and $S_K$ is defined in \eqref{eq: SK}. Note that $\tilde\a_F(x)$ is a constant on $F$ when $F \notin \Ei$. The global error estimator $\eta$ is then defined by
\beq\label{est}
\eta=\left( \sum_{K \in \cT} \eta_K^2\right)^{1/2}.
\eeq

\begin{remark}
 If $K$ is a non-interface element, i.e., $\cE_K \cap \Ei = \emptyset$, the first and second terms in \eqref{eta-K} vanish. 
In this case, the local error indicator $\eta_K$ is identical to the residual-based error indicator for 
the classical body-fitting conforming finite element method \cite{BeVe:00}.
\end{remark}

\section{Global Reliability}\label{sec:4}
In this section, we establish the reliability bound of the global estimator $\eta$ given in \eqref{est}.
For each $z \in \cN$, let $\o_z$ be the union of all elements sharing
$z$ as a common vertex. 
To this end, let $\cN_\Gamma$ be the set of all vertices $z$ such that 
$\mbox{meas}_{d-1} (\o_z \cap \Gamma) >0 $, and define
\begin{eqnarray*}
	H_{f}(\cT) &=& 
	\left(  \sum_{z \in \cN \setminus (\cN_\Gamma \cup \cN_D)}  
	\dfrac{\mbox{diam}(\o_z)^2}{\a_z} \|f -f_z\|^2_{0,\o_z} \right. \\[2mm]
	&&
	\left. \quad + \sum_{z \in \cN_\Gamma \cup \cN_D} 
	 \left(  \dfrac{\mbox{diam}(\o_z)^2}{\a^-}\|f\|_{0,\o_z^-}^2 + 
	 \dfrac{\mbox{diam}(\o_z)^2}{\a^+}\|f\|_{0,\o_z^+}^2\right) \right)^{1/2}
\end{eqnarray*}
where $\a_z= \a(x)|_{\o_z}$ and
$f_z$ is the average value of $f$ on $\o_z$.

\begin{remark}
The first term in $H_f(\cT)$ is a higher order term for $f \in L^2(\O)$  \cite{CaVe:99}. 
It is also well known that for linear finite element methods the edge residuals are dominant. In our adaptive algorithm the element residual is also omitted.
\end{remark}

\subsection{Helmholtz decomposition}
Let  
\[
H_N^1(\O)=
	\left\{ v \in H^1(\O) \,: \int_\O v \,dx =0 \quad\mbox{and} \quad \dfrac{\partial v}{  \partial t} = 0
	 \mbox{ on } \Gamma_N
	 \right\}.
\]
For $\phi \in H^1(\O)$, define 
the adjoint curl operator  by
$\gperp \, \phi = \left(- \dfrac{\partial \phi}{\partial y}, \,\dfrac{\partial \phi}{\partial x} \right)$.
For each $v \in \mathcal{S}(\cT)$, we define the discrete gradient operator $\nabla_h$ by
\[
	(\nabla_h v)|_K = \nabla (v|_K), \quad \forall \, K \in \cT.  
\]

\begin{lemma}[Helmholtz Decomposition] \label{HD}
Let $u$ and $u_\cT$ be the solutions of \eqref{continuous-weak-solution} and \eqref{IFE approx}, respectively. 
Then there exist uniquely $\phi \in H_{D}^1(\O)$ and $\psi \in H_N^1(\O)$ 
such that
\beq \label{HD:1}
	\a \nabla u - \tilde \a \nabla_h u_\cT=\a \nabla \phi+\gperp\,\psi.
\eeq
Moreover,
\beq \label{HD:3}
( \nabla \phi, \gperp \psi)=0.
\eeq

\end{lemma}
\begin{proof}
The proof can be referred to \cite{DaDuPaVa:1996, AinsworthRankin:08}. Here we also sketch a proof for the convenience of readers. Let $\tilde \be_\sigma := \a \nabla u - \tilde \a \nabla_h u_\cT$ and $e_\cT = u - u_\cT$.
Assume that $\phi \in H^1(\O)$ solves the following equation:
\[
	- \nabla \cdot \a \nabla \phi = -\nabla_h \cdot \tilde \be_\sigma \quad \mbox{in } \O
\]
with the boundary conditions
\[
	\phi = 0 \quad \mbox{on }\Gamma_D \quad \mbox{and} \quad
	-\a \nabla \phi \cdot \bn = -\tilde \be_\sigma \cdot \bn \quad \mbox{on } \Gamma_N.
\]

Then there exists $\psi \in H^1(\O)$ such that 
\[
	\nabla^\perp \psi = \tilde \be_\sigma - \a \nabla \phi.
\]
Moreover, we have
$ \dfrac{\partial \psi}{\partial t}= \nabla^\perp \psi  \cdot \bn =(\tilde \be_\sigma - \a \nabla \phi) \cdot \bn=0$ on $\Gamma_N$, hence, $\psi \in H_N^1(\O)$.
By integration by parts and the boundary conditions it is easy to check that
\[
	(\nabla \phi, \nabla^\perp \psi) =0.
\] 
This completes the proof of the lemma.
\end{proof}

\begin{lemma} \label{L2 representation with Helmholtz Decomposition}
Let $\phi$ and $\psi$ be given in \eqref{HD:1}.
Then we have the following error representations in the weighted semi-$H^1$ norm:
\begin{equation} \label{p-representation}
\begin{split}
	\|\a^{1/2} \nabla \phi\|_{0,\O}^2&=
	 (f, \phi-v)
	-\sum_{F\in \cE_I \cup \cE_N} \int_F j_{n,F} \{\phi-v\} ds
	\\
	&+\epsilon \sum_{F\in \Ei} \int_F \{\tilde \a \nabla v \cdot \bn_F \} \jump{u_\cT} \,ds
		+
	\sum_{F\in \Ei}  \dfrac{\gamma}{h_F}\int_F\tilde  \a\, \jump{u_\cT}\jump{v} \,ds, 
\end{split}
\end{equation}
for any  $v \in  \mathcal{S}_D(\cT)$ and
\begin{equation}\label{q-representation}
	\|\a^{-1/2} \gperp \psi \|_{0,\O}^2=
	-\sum_{F\in \Ei} \int_F \jump{u_\cT} \big( \gperp \psi \cdot \bn_F \big) ds +
	 ((1 - \tilde \a/\a) \nabla_h u_\cT, \gperp \psi). 
\end{equation}
\end{lemma}

\begin{proof}
Let $v \in \mathcal{S}_D(\cT)$ be arbitrary. Applying  \eqref{HD:3}, \eqref{HD:1}, 
and integration by parts gives 
\begin{equation} \label{L20}
\begin{split}
	&(\a \nabla \phi,\nabla \phi)
	=(\tilde \be_\sigma, \nabla \phi)
	=\Big(\tilde \be_\sigma,\nabla_h ( \phi- v) \Big)
	+\big(\tilde \be_\sigma, \nabla_h v \big)
	\\
  	=& \sum_{K\in \cT} \left( \int_K \big(f,\phi-v \big) \,dx
	+  \int_{\partial K} \big(\a \nabla u \cdot \bn \big) \big(\phi-v \big) \,ds\right) 
	\\
	&-  \sum_{K \in \cT} \left(\tilde \a\nabla u_\cT, \nabla_h ( \phi- v) \right)_K
	+\big(\tilde \be_\sigma, \nabla_h v \big)
	\\
	=&
	(f, \phi -v)  -\sum_{F \in \cE^{int}} \int_F (\a \nabla u \cdot \bn_F) \jump{v} \,ds
	- \sum_{F \in \cE_N} \int_F g_N ( \phi -v) \,ds 
	\\
	&-  \sum_{K \in \cT} \left(\tilde \a\nabla u_\cT, \nabla_h ( \phi- v) \right)_K
	+\big(\tilde \be_\sigma, \nabla_h v \big).
	\end{split}
\end{equation}
Applying integration by parts again gives
\begin{equation} \label{L21}
\begin{split}
	&\sum_{K \in \cT} \left(\tilde\a\nabla u_\cT, \nabla_h ( \phi- v) \right)_K 
	=
	\sum_{K \in \cT} \int_{\partial K} (\tilde\a \nabla u_\cT) \cdot \bn  ( \phi- v)  \,ds
	\\
	=&
	\sum_{F \in \cE_I } \int_F \jump{ \tilde\a \nabla u_\cT \cdot \bn_F} \{\phi-v\} \,ds
	-\sum_{F \in \cE^{int}} \int_F \{\tilde\a \nabla u_\cT \cdot \bn_F \} \jump{v} \,ds 
	\\
	&+
	\sum_{F \in \cE_N} \int_F (\tilde \a \nabla u_\cT \cdot \bn_F) (\phi-v) \,ds.
\end{split}
\end{equation}
The last equality used 
\eqref{jump-formula}, $\phi-v =0$ on $\Gamma_D$, and the facts that
 \[
 	\jump{v}_F=0, \quad \forall \, F \in \cE_I \setminus \Ei \quad \mbox{and} \quad
	\jump{\phi}_F =0, \quad \forall \, F \in \cE_I.
	\]
By integration by parts and \eqref{IFE approx} we also have
\begin{equation} 
\begin{split}
\big(\tilde \be_{\sigma} , \nabla_h v \big)
=&
(\a \nabla u, \nabla_h v) - (\tilde \a \nabla_h u_\cT, \nabla_h v) 
\\
=&\sum_{F\in \Ei} \int_F \{ \a \nabla u \,\cdot\, \bn_F\} \jump{v} \,ds
-
\sum_{F\in \Ei} \int_F \{ \tilde \a \nabla u_\cT \,\cdot\, \bn_F\} \jump{v} \,ds
\\&
+\epsilon \sum_{F\in \Ei} \int_F \{ \tilde\a \nabla v \,\cdot\, \bn_F\} \jump{u_\cT} \,ds
	+\sum_{F\in \Ei} \dfrac{\gamma }{h_F}\int_F \tilde\a \, \jump{u_\cT}\jump{v} \,ds,
\end{split}
\end{equation}
which, together with \eqref{L20} and \eqref{L21}, gives  \eqref{p-representation}.

To prove \eqref{q-representation}, by \eqref{HD:3}, \eqref{HD:1}, integration by parts, and the facts that
\[
	\jump{e_\cT}_F = - \jump{u_\cT}_F 
	\quad \mbox{and} \quad
	\jump{ \gperp \psi \cdot \bn_F}_F = 0, \quad \forall \,F \in \cE_I,
\]
we have
\begin{equation*}
\begin{split}
	&(\a^{-1} \gperp \psi, \gperp \psi)=( \nabla u - (\tilde \a/\a) \nabla_h u_\cT, \gperp \psi) 
	 \\
	 =& ( \nabla u - \nabla_h u_\cT, \gperp \psi) + ((1 - \tilde \a/\a) \nabla_h u_\cT, \gperp \psi)
	 \\
	=&\sum_{K \in \cT} \int_{\partial K} e_\cT (\gperp \psi \cdot \bn) \, ds
	+ ((1 - \tilde \a/\a) \nabla_h u_\cT, \gperp \psi)\\
	= &- \sum_{F \in \Ei}\int_F \jump{u_\cT} \big( \gperp \psi \cdot \bn_F \big)\,ds
	+ ((1 - \tilde \a/\a) \nabla_h u_\cT, \gperp \psi).
\end{split}
\end{equation*}
Hence, we obtain \eqref{q-representation}. This completes the proof of the lemma.
\end{proof}

\subsection{Modified Cl\'ement-type interpolation}
Define a modified Cl\'ement-type interpolation operator $\mathcal{I}_h\,:\, H_{D}^1(\O) \rightarrow S_{D}(\cT)$ by
\beq \label{cle:1}
	\mathcal{I}_h(v)=\sum_{z \in \cN } (\pi_z v) \tilde\lambda_z(x)
\eeq
where  $\pi_z$ is defined by
\[
\pi_z(v)= \left\{ 
	\begin{array}{ll}
		\dfrac{\int_{\o_z} \lambda_z v  \,dx}{ \int_{\o_z} \lambda_z \,dx}, & \forall \, z \in \cN \setminus \cN_D, \\[4mm]
		0,& \forall \, z \in \cN_D,
	\end{array}
\right.
\]
where $\lambda_z $  and $\tilde \lambda_z$ are the classical barycentric hat function and the linear IFEM nodal basis function
of $\mathcal{S}(\cT)$ associated to $z$, respectively. Note that
\beq \label{weighted-interpolation}
	(v-\pi_z v, \lambda_z)_{\o_z}=0 \quad \forall \, z \in \cN \setminus \cN_D.
\eeq
By Lemma~6.1 in \cite{CaVe:99} there holds
\beq\label{approximation}
	\| v - \pi_z v\|_{0,\o_z} \le C \mbox{diam}(\o_z) \| \nabla v\|_{0,\o_z}, \quad \forall \, z\in \cN  \quad \mbox{and} \quad
	\forall \,v \in H_D^1(\O).
\eeq

The following lemma provides the approximation and stability properties of the modified Cl\'ement-type interpolation operator.

\begin{lemma}[Cl\'ement-type Interpolation]\label{lem:clement}
Let $v \in H_{D}^1(\O)$, and $\mathcal{I}_h v \in \mathcal{S}_D(\cT)$ be the interpolation of $v$ defined in \eqref{cle:1}. 
Then there exists a constant  $C>0$ that is independent of the mesh size and the location of the interface such that
\begin{align}
	 \big\|v-\mathcal{I}_h v \big\|_{0,K} &\le C h_K \big\|\nabla v\big\|_{0,\o_K},
	 \quad \forall \, K  \in \cT, \label{cle:2}
	 \\ 
	\big\|\nabla (v-\mathcal{I}_h v) \big\|_{0,K} &\le C  \big\|\nabla v \big\|_{0,\o_K},
	\quad \quad \;\forall \,K  \in \cT, \;   \label{cle:3}
	\\
	\big\| (v-\mathcal{I}_h v)|_{K} \big\|_{0,F} &\le C h_F^{1/2} \big\|\nabla v\big\|_{0,\o_{K}},
	\; \; \forall  \,F \in \cE_K,    \;K \in \{K_{F,1}, K_{F,2}\} , \label{cle:4}
\end{align}
where 
$\o_K$ is the union of all elements sharing at least one vertex with $K$.
\end{lemma}

\begin{proof}
	Without loss of generality, assume that $K \in \cT$ is an interior element.
	To prove \eqref{cle:2}, by the partition of unity, the triangle inequality,  the boundedness of IFE basis functions  
	$\|\tilde\lambda_z\|_{\infty,K} \le C(\rho)$ (Theorem 2.4, \cite{LiLinLinRo:04}), and \eqref{approximation},
	we have
	\begin{eqnarray*}\label{cle:2a}
	\big\|v-\mathcal{I}_h v \big\|_{0,K} &\le& \sum_{z\in \cN_K} \big\| (v- \pi_zv) \tilde\lambda_z \big\|_{0,K} 
	\le C \sum_{z\in \cN_K} \big\| v- \pi_z v \big\|_{0,K} \\[2mm]
	& \le& C \sum_{z\in \cN_K} \big\| v- \pi_z v \big\|_{0,\o_z} \le 
	C \sum_{z\in \cN_K} h_K \| \nabla v \|_{0,\o_z} \le C h_K \|\nabla v\|_{0,\o_K}.
	\end{eqnarray*}
	
	To prove \eqref{cle:3}, by the partition of unity, the triangle inequality, the fact that 
	$\|\nabla \tilde \lambda_z\|_{\infty,K} \le C h_K^{-1}$ (Theorem 2.4, \cite{LiLinLinRo:04}),  and \eqref{approximation}, we have
	\begin{eqnarray*}
		\big\|\nabla (v-\mathcal{I}_h v) \big\|_{0,K} &=&
		\left\| \sum_{z \in \cN_K} \nabla \Big((v-\pi_z v)\tilde\lambda_z \Big) \right\|_{0,K}
			\le \sum_{z \in \cN_K} \Big( \big\|\tilde\lambda_z \nabla v \big\|_{0,K} 
			+ \big\|(v-\pi_z v) \nabla \tilde\lambda_z \big\|_{0,K} \Big)\\[2mm]
		&\le& C \left( \big\|\nabla v \big\|_{0,K} + h_K^{-1}\big\|v-\pi_z v \big\|_{0,K}	 \right) 
		\le C \big\|\nabla v \big\|_{0,\o_K}.
	\end{eqnarray*}
	
	Finally, \eqref{cle:4}
	follows from the partition of unity, the triangle inequality, the trace inequality, and \eqref{approximation}:
	\begin{eqnarray*}
		\big\|(v-\mathcal{I}_h v)|_K \big\|_{0,F} &\le& \sum_{z\in \cN_K} \big\|(v-\pi_z v) \tilde \lambda_z|_K\big\|_{0,F}
		\le C \sum_{z\in \cN_K} \big\|(v-\pi_z v)|_K\big\|_{0,F}
		\\[2mm]
		&\le& C \sum_{z \in \cN_K} 
			\left(h_F^{-1/2} \big\| v-\pi_z v \big\|_{0,K} 
			+h_F^{1/2}\big \|\nabla v \big\|_{0,K}\right)	
			\le C h_F^{1/2}\big \|\nabla v \big\|_{0,\o_{K}}.
	\end{eqnarray*}
	This completes the proof of the lemma.
	\end{proof}

\begin{lemma}
There exists a constant $C>0$, independent of the mesh size
and the location of the interface, such that
\beq \label{norm-0-1}
	\| \jump{u_\cT}\|_{0,F} \le C h_F\| j_{t,F}\|_{0,F}, \quad \forall \,F \in \Ei
\eeq
and
\beq\label{norm-1/2-1}
	\| \jump{u_\cT}\|_{1/2,F} \le h_F^{1/2}\|j_{t,F}\|_{0,F}, \quad \forall \,F \in \Ei.
\eeq
\end{lemma}

\begin{proof}
We first prove the results on the reference element $K$ formed by vertices
$((0,0), (1,0), (0,1))$ and let $F$ be the edge of $K$ on the $x$-axis. Without loss of generality, let
$(a,0), 0< a < 1$, be the interface point on $F$ and $\jump{u_\cT}_F(a,0) =b$.
Note that $\jump{u_\cT}_F$ takes the value $0$ at both endpoints $(0,0)$ and $(1,0)$. By direct calculations, we have
\[	
	\|\jump{u_\cT}\|_{0,F}^2 = \dfrac{1}{3}b^2 \quad \mbox{and} \quad
	\| \jump{\nabla u_\cT \cdot \bt}\|_{0,F}^2 
	= \left(\dfrac{1}{a} + \dfrac{1}{1-a}\right) b^2 \ge 4b^2.
\]
Regarding the $H^{1/2}$-norm, we have 
\[
	\|\jump{u_\cT}\|_{1/2,F} = \inf_{v \in \tilde H^1(K)} \|v\|_{1,K}
\]
where 
$\tilde H^1(K) = 
	\{v \in H^1(K) : v = \jump{u_\cT} \,~\mbox{on}\, F, ~v = 0\, ~\mbox{on } \,\partial K \setminus \{F \}$.
In particular, let
\[
	v = \begin{cases}
		\dfrac{b}{a}x &\mbox{ in } K_1 \\
		\dfrac{b}{a-1} (y+x -1) &\mbox{ in } K_2
	\end{cases}
\]
where $K_1$ is the subtriangle fromed by $((0,0),(a,0),(0,1))$ and $K_2 = K \setminus K_1$. It is easy to verify that $v \in \tilde H^1(K)$.
Another direct calculation gives
\[
	\|v\|_{1,K}^2 = \dfrac{1}{12}b^2 + \left( \dfrac{1}{2a} + \dfrac{1}{1-a}\right)b^2 .
\]
It is easy to verify that
\[
	\| \jump{u_\cT}\|_{0,F} \le  \| j_{t,F}\|_{0,F} \quad\mbox{ and} \quad
	\| \jump{u_\cT}\|_{1/2,F} \le  \|v\|_{1,K} \le 2\|j_{t,F}\|_{0,F},
\]
which, together with the scaling argument, gives \eqref{norm-0-1} and \eqref{norm-1/2-1}. This completes the proof of the lemma.
\end{proof}

\begin{lemma} \label{lemma:con}
Let $\phi$ be given in \eqref{HD:1}. There exists a constant $C$ independent of the mesh size and the 
location of the interface such that
\begin{equation}\label{rel:con}
	\|\a^{1/2} \nabla \phi\|_{0,\O} \le
	 C \left( 
	 \sum_{F\in \cE_I \cup \cE_N}   h_F 
	\big\|\tilde \a_F^{-1/2}j_{n,F} \big\|_{0,F}^2 
	+
	 \sum_{F\in \cE^{int}}    h_F
	\big\| \tilde \a^{1/2} j_{t,F} \big\|_{0,F}^2  
	+H_f(\cT)^2
	\right)^{1/2}.
\end{equation}
\end{lemma}

\begin{proof}
By Lemma \ref{L2 representation with Helmholtz Decomposition},
\begin{eqnarray*} 
	\|\a^{1/2} \nabla \phi\|_{0,\O}^2
	&=&
	 (f, \phi-v)
	-\sum_{F\in \cE_I \cup \cE_N} \int_F j_{n,F} \{\phi-v\} ds
	 \\[2mm]
	&&~+ \epsilon \sum_{F\in \Ei} \int_F \{ \tilde \a \nabla v \cdot \bn_F \} \jump{u_\cT} \,ds
		+
	\sum_{F\in \Ei}  \dfrac{\gamma}{h_F}\int_F \tilde \a\, \jump{u_\cT}\jump{v} \,ds\\[2mm]
	&
	 \triangleq&  I_1+I_2+I_3+I_4.
\end{eqnarray*}
Let $v= \phi_I \in \mathcal{S}_D(\cT)$ be the modified Cl\'ement-type interpolation defined in \eqref{cle:1} of $\phi$.
By the partition of unity,  the fact that $\tilde \lambda_z = \lambda_z$ for $z \in \cN\setminus \cN_\Gamma$,
the Cauchy-Schwarz inequality, \eqref{weighted-interpolation}, and \eqref{approximation}, we have
\begin{equation}\label{I-1}
\begin{split}
	I_1
	& =
	\sum_{z \in \cN_\Gamma \cup \cN_D} \bigg(f\tilde\lambda_z , \phi- \pi_z \phi\bigg)_{\o_z}
	+
	\sum_{z \in \cN \setminus (\cN_\Gamma \cup \cN_D)} \bigg((f - f_z) \lambda_z,  \phi- \pi_z \phi\bigg)_{\o_z}
	\\
	&\le
	C\left(  \sum_{z \in  \cN_\Gamma \cup \cN_D} \mbox{diam}(\o_z) \big\| f \big\|_{0,\o_z} \big\|\nabla \phi \big\|_{0,\o_z}
	+
	\sum_{z \in \cN \setminus (\cN_\Gamma \cup \cN_D)} \mbox{diam}(\o_z) \big\| f-f_z \big\|_{0,\o_z} \big\|\nabla \phi \big\|_{0,\o_z}
	 \right) 
	 \\
	&\le C H_f(\cT) \big\|\a^{1/2} \nabla \phi \big\|_{0,\O}.
\end{split}
\end{equation}
For any $F \in (\cE_I \cup \cE_N) \setminus \Ei$,
it follows from the continuity of $\phi-\phi_{I}$ on $F$, Cauchy-Schwarz inequality, and \eqref{cle:4} that
\begin{equation}\label{I_2b}
\begin{split}
	\int_F j_{n,F} \{\phi- \phi_{I}\}\,ds=&\int_F j_{n,F} (\phi- \phi_{I})|_{K_{F,1}}\,ds 
	\le C h_F^{1/2} \big\|  j_{n,F} \big\|_{0,F} \big\|\nabla \phi \big\|_{0,\o_{K_{F,1}}} 
	\\
	\le& C h_F^{1/2} \big\|  \a_F^{-1/2} j_{n,F} \big\|_{0,F} \big\| \a^{1/2} \nabla \phi \big\|_{0,\o_{K_{F,1}}}.
\end{split}
\end{equation}
For any $F \in \Ei$, it then follows from the Cauchy-Schwarz and Young's inequality and \eqref{cle:4} that
\begin{eqnarray*}\label{I_2a}
	&&\int_F j_{n,F} \{ \phi-\phi_{I}\} \,ds \\\
	&=&
	 \int_{F+} j_{n,F} \{ \phi-\phi_{I}\} \,ds +\int_{F^-} j_{n,F} \{ \phi-\phi_{I}\} \,ds\\[2mm]
	 &\le&
	\|j_{n,F}\|_{0,F^+}
	\bigg( \big\|(\phi-\phi_{I})|_{K_{F,1}} \big\|_{0,F^+}+ \big\|(\phi-\phi_{I})|_{0,K_{F,2}} \big\|_{0,F^+}\bigg)\\[2mm]
	&&+
		\|j_{n,F}\|_{0,F^-}
	\bigg( \|(\phi-\phi_{I})|_{K_{F,1}} \big\|_{0,F^-}+ \|(\phi-\phi_{I})|_{K_{F,2}} \big\|_{0,F^-} \bigg)\\[2mm]
	&\le& 
	C \left( \dfrac{1}{\a^+} \|j_{n,F}\|_{0,F^+}^2 
	\!+\!   \dfrac{1}{\a^{-}} \|j_{n,F}\|_{0,F^-}^2  \right)^{1/2}
	\!\! \left( 
	\sum_{K \in \{K_{F,1}, K_{F,2}\}} \big\|\alpha^{1/2}(\phi-\phi_{I})|_{K} \big\|^2_{0,F}
	 \right)^{1/2}\\[2mm]
	&\le&  C h_F^{1/2} \big\|  \tilde\a_F^{-1/2} j_{n,F}\big\|_{0,F}
	 \| \a ^{1/2} \nabla \phi\|_{0,\o_F},
\end{eqnarray*}
where $\o_F = \o_{K_{F,1}} \cup \o_{K_{F,2}}$, which, together with
\eqref{I_2b} and the Cauchy-Schwarz inequality, yields 
\beq \label{I-2}
	I_2 \le 
	C\left(  \sum_{F \in \cE_I \cup \cE_N}  h_F \big\|  \tilde\a_F^{-1/2} j_{n,F}\big\|_{0,F}^2 \right)^{1/2} \|\a^{1/2} \nabla \phi\|_{0,\O}.
\eeq
To bound $I_3$, we apply the following trace inequality 
for functions in the IFEM space (see Lemma 3.2 in \cite{LinLinZhang:15}).
Let $v \in \mathcal{S}_D(\cT)$ be arbitrary, then 
\begin{equation}\label{trace:n}
	 \big\|\tilde\a \nabla v|_K \cdot \bn_F \big\|_{0,F} \le 
	  C h_F^{-1/2}\big\|\tilde\a \nabla v \big\|_{0,K},\quad \forall \,F \in \Ei.
\end{equation}
It now follows  from the Cauchy-Schwarz inequality, \eqref{trace:n}, and 
 \eqref{cle:3} that
\begin{eqnarray*}
I_3	&\le&\sum_{F\in \Ei} \|\jump{u_\cT}\|_{0,F}
	\left( \left\| \tilde \a \nabla \phi_{I}|_{K_{F,1}} \cdot \bn_F \right\|_{0,F}
		+ \left\| \tilde \a \nabla \phi_{I}|_{K_{F,2}}\cdot \bn_F \right\|_{0,F}\right)
		 \\[2mm]
	&\le& \sum_{F\in \Ei} C \, h_F^{-1/2} \|\jump{u_\cT}\|_{0,F} 
	\bigg( \left \| \tilde\a \nabla \phi_{I} \right \|_{0,K_{F,1}}
		+ \left\| \tilde\a \nabla \phi_{I} \right\|_{0,K_{F,2}}
	\bigg) \\[2mm]
	&\le&C \left( \sum_{F\in \Ei}   h_F^{-1} 
	\| \tilde \a_F^{1/2}\jump{u_\cT}\|_{0,F}^2 \right)^{1/2}
	\|\a^{1/2}\nabla \phi\|_{0, \O}.
\end{eqnarray*}
Together with \eqref{norm-0-1} we obtain
\beq\label{I-3}
I_3	\le
	C \left( \sum_{F\in \Ei}   h_F 
	\|\tilde \a_F^{1/2} j_{t,F}\|_{0,F}^2 \right)^{1/2}
	\|\a^{1/2}\nabla \phi\|_{0,\O}.
\eeq
To bound $I_4$, we use the fact that
\[
	\int_F \tilde\a \jump{u_\cT} \jump{\phi_{I}} \,ds = 
	-\int_F \tilde \a \jump{u_\cT} \jump{\phi- \phi_{I} }\,ds
	\quad \forall \,F \in \Ei,
\]
 the Cauchy-Schwarz and triangle inequalities, and \eqref{cle:4} to obtain
\begin{eqnarray*}
	I_4
	&=&-\sum_{F\in \cE^{int}}\dfrac{\gamma}{h_F} \int_F \tilde \a \,\jump{u_\cT} \jump{\phi-\phi_{I}} \,ds\nonumber\\[2mm]
	&\le& C \sum_{F\in \cE^{int}} \dfrac{\gamma  }{h_F}  \|\tilde \a_F^{1/2} \jump{u_\cT}\|_{0,F}
		\left( \|\a^{1/2}(\phi-\phi_{I})|_{K_{F,1}} \|_{0,F}+ \|\a^{1/2}(\phi-\phi_{I})|_{K_{F,2}} \|_{0,F}\right)\nonumber\\[2mm]
		&\le&
		 C \bigg( \sum_{F\in \cE^{int}}  h_F^{-1} 
	\|\tilde \a_F^{1/2} \jump{u_\cT}\|_{0,F}^2 \bigg)^{1/2}
	\|\a^{1/2}\nabla \phi\|_{0,\O},
\end{eqnarray*} 
which, combining with \eqref{norm-0-1}, yields
\beq\label{I-4}
I_4	\le C \left( \sum_{F\in \cE^{int}}  h_F
	\|\tilde \a_F^{1/2} j_{t,F}\|_{0,F}^2 \right)^{1/2}
	\|\a^{1/2}\nabla \phi\|_{0,\O}.
\eeq

Finally, \eqref{rel:con} is a direct consequence of \eqref{I-1}, \eqref{I-2}, \eqref{I-3}, \eqref{I-4} and the
Young's inequality.
This completes the proof of the lemma.
\end{proof}

\begin{lemma} \label{lemma:noncon}
Let $\psi$ be given in \eqref{HD:1}. There exists a constant $C$ independent of the mesh size and the location of the interface such that
\beq\label{nonconformist-estimate}
	\|\a^{-1/2} \gperp \psi\|_{0,\O} \le  C \bigg(  \sum_{F \in \cE^{int}}  h_F
	\| \tilde\a_F^{1/2} j_{t,F}\|_{0,F}^2
	+
	\sum_{K \in \cT^{int}} \| \tilde \a^{1/2} \nabla u_\cT\|_{0,S_K}^2
	 \bigg)^{1/2}.
\eeq
\end{lemma}

\begin{proof}
For the first term in \eqref{q-representation} it follows from the duality inequality, \eqref{norm-1/2-1}, the trace and Young's inequalities that
\begin{equation}\label{I-21}
\begin{split}
	&\sum_{F \in \cE^{int}} \int_F \jump{u_\cT} (\gperp \psi \cdot \bn_F) \,ds
	\le \sum_{F \in \cE^{int}} \big\| \jump{u_\cT} \big\|_{1/2,F} 
	\left\|\gperp \psi \cdot \bn_F \right\|_{-1/2,F}
	\\
	\le~&C\sum_{F \in \cE^{int}}  h_F^{1/2}
	 \big\| j_{t,F} \big\|_{0,F} \big\|\gperp\psi \big\|_{0,K_{F,1}}
	 \\
	\le~&  C \bigg(  \sum_{F \in \cE^{int}}  h_F
	\| \tilde \a_F^{1/2} j_{t,F}\|_{0,F}^2 \bigg)^{1/2}
	 \|\a^{-1/2} \gperp \psi\|_{0,\O}.
\end{split}
\end{equation}
For the second term in \eqref{q-representation} it follows from the Cauchy-Schwarz inequality that
\begin{equation}\label{I-22}
\begin{split}
&((1 - \tilde \a/\a) \nabla u_\cT, \gperp \psi) \le C
\sum_{K \in \cT^{int}}  \| \tilde\a^{1/2}\nabla u_\cT \|_{0,S_K}  \|\a^{-1/2} \gperp \psi\|_{0,S_K}
\\
\le~ &C \left( \sum_{K \in \cT^{int}} \| \tilde \a^{1/2} \nabla u_\cT\|_{0,S_K}^2 \right)^{1/2} 
\|\a^{-1/2} \gperp \psi\|_{0,\O}.
\end{split}
\end{equation}
Combining \eqref{I-21}-\eqref{I-22} and \eqref{q-representation} gives \eqref{nonconformist-estimate}.
 This completes the proof of the lemma.
\end{proof}

\begin{theorem}\label{global-reliability}{\em(Global Reliability)}
There exists a constant $C_r>0$ that is independent of the location of the interface and the mesh size, such that
\beq\label{reliability}
	\|\a^{1/2}  ( \nabla u - \nabla_hu_\cT) \|_{0,\O} \le C_r (\eta +H_f(\cT)).
\eeq

\end{theorem}

\begin{proof}
First by adding and subtracting proper terms we have
	\begin{equation} \label{rel:1}
		\begin{split}
			&\| \a^{1/2} ( \nabla u - \nabla_hu_\cT)\|_{0,\O}^2\\
			=&
			(\a \nabla u - \tilde \a \nabla_h u_\cT, \nabla u - \nabla_hu_\cT)
			+( (\tilde \a - \a)\nabla_h u_\cT, \nabla u - \nabla_hu_\cT )
			\\=&
	(\a \nabla u - \tilde \a \nabla_h u_\cT, \a^{-1}(\a\nabla u - \tilde \a \nabla_h u_\cT) )	
	\\&+
	(\a \nabla u - \tilde \a \nabla_h u_\cT,(\tilde \a / \a -1)\nabla_h u_\cT) )+
	( (\tilde \a - \a)\nabla_h u_\cT, \nabla u - \nabla_hu_\cT )
	\\=&
	(\a \nabla u - \tilde \a \nabla_h u_\cT, \a^{-1}(\a\nabla u - \tilde \a \nabla_h u_\cT) )	
	+
	2( \nabla u - \nabla_hu_\cT,(\tilde \a  -\a)\nabla_h u_\cT) 
	\\&
	-
	(\a^{-1}(\tilde \a  -\a)\nabla_h u_\cT,(\tilde \a -\a)\nabla_h u_\cT) )\\
	\le&
	(\a \nabla u - \tilde \a \nabla_h u_\cT, \a^{-1}(\a\nabla u - \tilde \a \nabla_h u_\cT) )	
	+
	2(\nabla u - \nabla_hu_\cT,(\tilde \a  -\a)\nabla_h u_\cT).
		\end{split}
	\end{equation}	
By \eqref{HD:1} and \eqref{HD:3} we have
	\begin{equation} \label{rel:2}
	\begin{split}
	&(\a \nabla u - \tilde \a \nabla_h u_\cT, \a^{-1}(\a\nabla u - \tilde \a \nabla_h u_\cT) )	
	=(\a \nabla \phi + \gperp \psi, \nabla \phi +\a^{-1} \gperp \psi)\\
	=&\|\a^{1/2} \nabla \phi\|_{0,\Omega}^2 + \|\a^{-1/2} \gperp \psi\|_{0,\Omega}^2.
	\end{split}
	\end{equation}
Applying the Cauchy-Schwarz and Young's inequalities also gives	
\begin{equation}\label{rel:3}
(\nabla u - \nabla_hu_\cT,(\tilde \a  -\a)\nabla_h u_\cT) \le C
\left( \sum_{K \in \cT^{int}} \|\tilde \a \nabla_h u_\cT \|_{0,S_K}^2 \right)^{1/2}
\|\a^{1/2} ( \nabla u - \nabla_hu_\cT)\|_{0,\Omega}.
\end{equation}
\eqref{reliability} is then a direct result of \eqref{rel:1}--\eqref{rel:3},  Lemma \ref{lemma:con}, Lemma \ref{lemma:noncon}
and the Young's inequality. This completes the proof of theorem.
\end{proof}

\section{Local Efficiency}\label{sec:5}
In this section, we establish the efficiency bound for the error indicators $\eta_K$ defined in 
\eqref{eta-K} for every element $K \in \cT$. For non-interface elements, the efficiency bound for $\eta_K$ is well known (see \cite{BeVe:00}) and the key technique is using the local edge and element bubble functions. 
However, the same technique without modification becomes invalid for interface elements because jumps of the normal flux and the tangential derivative on the interface edges become piecewise constant. To overcome this difficulty it is natural to design more localized bubble functions that allow us to derive the efficiency bounds on $F^+$ and $F^-$ separately for each $F \in \Ei$.

For each $F \in \Ei$ we first design auxiliary elements and edge bubble functions associated to $F^-$.
Choose $\tilde K_{F,1} \subset K_{F,1}$ and 
$\tilde K_{F,2} \subset K_{F,2}$ to be two regular triangular sub-elements such that
$F^-$ is a common edge of both $\tilde K_{F,1}$ and $\tilde K_{F, 2}$, to be more precise, 
\[
	\partial \tilde K_{F,1}  \cap \partial \tilde K_{F,2} = F^-
	\quad \mbox{and} \quad
	(\partial \tilde K_{F,1}  \cup \partial \tilde K_{F,2}) \cap F^+ =\emptyset.
\]
Note that $\tilde K_{F,1}$ and $\tilde K_{F,2}$ are not necessarily in $K_{F,1}^-$ and $K_{F,2}^-$, respectively. The key requirement here is to make sub-elements $\tilde K_{F,1}, \tilde K_{F,2}$ regular while  $K_{F,1}^-$ and $K_{F,2}^-$ may not be in general.
For each sub-element $\tilde K \in \{\tilde K_{F,1}, \tilde K_{F,2}\}$, define the auxiliary element bubble function $\upsilon_{\tilde K}$ such that
(i) $\upsilon_{\tilde K} \in P_3(\tilde K)$, 
(ii) $\upsilon_{\tilde K}|_{\partial \tilde K} \equiv 0$ and 
(iii) $\upsilon_{\tilde K}$ takes the value $1$ at the barycentric center of $\tilde K$. 
We also define the auxiliary edge bubble function $\upsilon_{F^-}$ for $F^-$ such that
(i) $\upsilon_{F^-}  |_{\tilde K} \in P_2(\tilde K)$ for $\tilde K \in \{ \tilde K_{F,1}, \tilde K_{F,2}\}$,
(ii) $\upsilon_{F^-}|_{\partial  (\tilde K_{F,1} \cup \tilde K_{F,2})} \equiv 0$ and
(iii) $ \upsilon_{F^-} $ takes value $1$ at the middle point of $F^-$.
It is easy to verify that $\upsilon_{\tilde K}$ for $\tilde K \in \{ \tilde K_{F,1}, \tilde K_{F,2}\} $ 
and $\upsilon_{F^-}$ uniquely exist. Let $w_{F^-} = (j_{n,F}|_{F^-})v_{F^-} $ and 
$w_{\tilde K} = v_{\tilde K}f_K$ with $f_K$ being the average of $f$ on $K$.

Applying the continuity of $\a \nabla u \cdot \bn_F$, the divergence theorem, and the Cauchy-Schwarz inequality gives
\begin{eqnarray*}
	\| j_{n,F}\|_{0,F^-}^2 &\le& C \int_{F^-} \jump{\tilde\a \nabla u_\cT \cdot \bn_F} w_{F^-} \,ds \\[2mm]
	&=&
	C\int_{F^-} \left(\jump{\tilde \a \nabla u_\cT \cdot \bn_F} - \jump{\a \nabla u \cdot \bn_F}\right) 	w_{F^-} \,ds\\[2mm]
	&=&
	C\int_{\tilde K_{F,1} \cup \tilde K_{F,2}} (\tilde \a \nabla_h u_\cT - \a \nabla u) \cdot \nabla w_{F^-} \,dx
	+\int_{\tilde K_{F,1} \cup \tilde K_{F,2}} f w_{F^-}\,dx\\[2mm]
	&\le& C
	\| \tilde \a \nabla_h u_\cT - \a \nabla u \|_{0,\tilde K_{F,1} \cup \tilde K_{F,2}} \| \nabla w_{F^-}\|_{0,\tilde 	K_{F,1} \cup \tilde K_{F,2}} +\|f\|_{0,\tilde K_{F,1} \cup \tilde K_{F,2}}\|w_{F^-}\|_{0,\tilde K_{F,1} \cup \tilde K_{F,2}},
\end{eqnarray*}
which, combining with the properties of $w_{F^-}$ that 
\[
	\|\nabla w_{F^-}\|_{0,\tilde K_{F,1} \cup \tilde K_{F,2}} \le C \dfrac{1}{h_F^-} \|w_{F^-}\|_{0,\tilde K_{F,1} \cup \tilde K_{F,2}} \le C\dfrac{1}{\sqrt{h_F^-}} \| j_{n,F}\|_{0,F^-},
\]
yields
\begin{equation} \label{effi-flux-jump}
	\| j_{n,F}\|_{0,F^-} \le C \left( \dfrac{1}{\sqrt{h_F^-}}\| \tilde \a \nabla_h u_\cT - \a \nabla u \|_{0,\tilde K_{F,1} \cup \tilde K_{F,2}} + \sqrt{h_F^-}\|f\|_{0,\tilde K_{F,1} \cup \tilde K_{F,2}}\right).
\end{equation}
We now establish the efficiency bound for the element residual $\|f\|_{0,\tilde K}$ for $\tilde K \in \{ \tilde K_{F,1}, \tilde K_{F,2}\}$. By the property of $w_{\tilde K}$, the triangle inequality, the
divergence theorem, and the Cauchy-Schwarz inequality, we have
\begin{eqnarray*}
	\|f_{K}\|_{0,\tilde K}^2 &\le& 
	C \left(\int_{\tilde K} f w_{\tilde K} \,dx + \| f - f_{K}\|_{0,\tilde K} \|w_K\|_{0,\tilde K} \right) \\[2mm]
	&\le& C \left( \int_{\tilde K} \nabla \cdot(\tilde \a \nabla u_\cT - \a \nabla u )w_{\tilde K} \,dx
	+ \| f - f_{K}\|_{0,\tilde K} \|w_{\tilde K}\|_{0,\tilde K} \right)\\[2mm]
	&\le&C \left( \|\tilde \a \nabla u_\cT - \a \nabla u \|_{0,\tilde K} \|\nabla w_{\tilde K}\|_{0,\tilde K}
	+\| f - f_{K}\|_{0,\tilde K} \|w_{\tilde K}\|_{0,\tilde K} 
	\right),
\end{eqnarray*}
which, combining with similar properties for $w_{\tilde K}$:
\[
	\| \nabla w_{\tilde K}\|_{0,\tilde K} \le C \dfrac{1}{h_F^-} \|w_{\tilde K}\|_{0,\tilde K} 
	\le C\dfrac{1}{h_F^-} \|f_{K}\|_{0,\tilde K},
\]
 gives
\beq \label{effi:elemen-residual}
	h_F^-\|f_{K}\|_{0,\tilde K} \le
	\left( \|\tilde \a \nabla u_\cT - \a \nabla u \|_{0,\tilde K}+ h_F^-\|f - f_K\|_{0,\tilde K}
	\right).
\eeq
By \eqref{effi-flux-jump} and \eqref{effi:elemen-residual}, we have
\[\sqrt{h_F^-}\|j_{n,F}\|_{0, F^-} 
	\le C \sum_{K \in \{K_{F,1}, K_{F,2}\}}
	\left( 
             \|  \a  \nabla u - \tilde \a \nabla u_\cT \|_{0,\tilde K}
             +  h_F^- \|f - f_K\|_{0,\tilde K} \right).
\]
Adding proper weights gives
\beq\label{effi:flux-jump-2}
	\sqrt{h_F^-}\|\tilde \a_F^{-1/2}j_{n,F}\|_{0,F^-} 
	\le  C  \sum_{K \in \{K_{F,1}, K_{F,2}\}} \!\!\!\!\left( 
             \|  \a^{1/2}  (\nabla u -\nabla u_\cT )\|_{0,K}
             +
              \|\tilde   \a^{1/2} \nabla u_\cT\|_{0,S_K}
             +  H_{f,K} \right)
\eeq
where
\[
	H_{f,K} = h_K \| \a^{-1/2}(f -f_K)\|_{0,K}, \quad \forall \, K \in \cT
\]
and the constant $C$ is independent of the mesh size and the location of the interface but may depend on $\rho$. 
Similarly, one can prove that
\[
\|j_{t,F}\|_{0,F^-} \le C \sum_{\tilde K \in \{\tilde K_{F,1}, \tilde K_{F,2}\}}  \dfrac{1}{\sqrt{h_F^-}} \|  \nabla u-\nabla u_\cT\|_{0,\tilde K}
\]
and, after adding proper weights, that
\beq\label{effi:tan-jump}
\sqrt{h_F^-} \|\tilde \a_F^{1/2}j_{t,F}\|_{0,F^-}  \le C \sum_{ K \in \{ K_{F,1},  K_{F,2}\}}  
\| \a^{1/2} (\nabla u-\nabla u_\cT)\|_{0,K}.
\eeq
Similarly, by defining auxiliary bubble functions on $F^+$, we also have the following local efficiency results on $F^+$: 
\beq\label{effi:flux-jump-3}
	\sqrt{h_F^+}\|\tilde \a_F^{-1/2}j_{n,F}\|_{0,F^+} \le C  \sum_{K \in \{K_{F,1}, K_{F,2}\}} \left( 
             \|  \a^{1/2} \nabla (u -u_\cT )\|_{K}
             +
              \|  \tilde \a^{1/2} \nabla u_\cT\|_{S_K}
             +  H_{f,K} \right)
\eeq
and
\beq\label{effi:tan-jump-2}
\sqrt{h_F^+} \|\tilde \a_F^{1/2}j_{t,F}\|_{0,F^+}  \le C \sum_{K \in \{K_{F,1}, K_{F,2}\}}  
\| \a^{1/2} (\nabla u-\nabla u_\cT)\|_{0,K}.
\eeq

For the first case we assume that for each $F \in \Ei$, there exist positive constants $\lambda$ and
$\Lambda$ such that
\beq\label{localtion-assump}
	\lambda \le \dfrac{h_{F}^+}{h_{F}^-} \le \Lambda.
\eeq

\begin{lemma}\label{lem:efficiency1}
Let $u$ and $u_\cT$ be the solution in \eqref{continuous-weak-solution} and \eqref{IFE approx}, respectively. 
Then there exists a constant $C_1$ that is independent of the mesh size but may depend on $\lambda, \Lambda,$ and 
$ \rho$ such that
	\beq \label{local-efficiency}
		\eta_K \le C_1 \left( \| \a^{1/2} (\nabla u- \nabla_h u_\cT) \|_{0,\o_K} 
		+  \sum_{K \subset \o_K} \| \tilde \a^{1/2} \nabla u_\cT\|_{0,S_K}
		+  \sum_{K \subset \o_K} H_{f,K}
		 \right),\; \forall \, K \in \cT.
	\eeq
\end{lemma}

\begin{proof} 
 In the case that K is not an interface element,  
\eqref{local-efficiency} is a well known result (see, e.g., \cite{CaHeZh-mcom:17}).
In the case that $K$ is an interface element, by \eqref{effi:flux-jump-2} and  \eqref{effi:flux-jump-3}, we have
\begin{eqnarray*}
	h_F^{1/2} \|\tilde a_F^{-1/2}j_{n,F}\|_{0,F} &=& 
	h_F^{1/2}  \left(  \|\tilde \a_F^{-1/2}j_{n,F}\|_{0,F^+}^2 +  \| \tilde \a_F^{-1/2}j_{n,F}\|_{0,F^-}^2\right)^{1/2} \\[2mm]
	&\le& 
	C_1\sum_{K \in \{ K_{F,1} , K_{F,2} \}} 
	\left( \| \a^{1/2} \nabla e_\cT\|_{0,K}+ \|\tilde \a^{1/2} \nabla u_{\cT}\|_{0,K} + H_{f,K} \right),
\end{eqnarray*}
where $C_1$ is independent of the mesh size but might depend on $\rho$, $\lambda$, and $\Lambda$.
Similarly, by \eqref{effi:tan-jump} and \eqref{effi:tan-jump-2}, we have
\beq
	h_F^{1/2} \|\tilde \a_F^{1/2}j_{t,F}\|_{0,F} \le C_1\sum_{K \in \{ K_{F,1} , K_{F,2} \}}  \| \a^{1/2} \nabla e_\cT\|_{0,K},
\eeq
where $C_1$ is independent of the mesh size but might depend on $\rho$, $\lambda$, and $\Lambda$.
This completes the proof of the lemma.
\end{proof}

In the above lemma the efficient constant $C_1$ blows up when the $\lambda$ and $\Lambda$ become extreme. However, we note that in the extreme circumstances the related basis functions for IFEM are very ``close" to classical FE nodal basis functions thanks to its consistency  with standard FE basis functions (see \cite{LiLinLinRo:04} for more detail). The partially penalized IFEM solution then should be also ``close" to 
the classical finite element solution on fitted meshes. In the following lemma we prove the efficiency in a different approach that yields a bounded coefficient constant when the  $\lambda$ and $\Lambda$ are extreme.

\begin{lemma}\label{lem:efficiency2}
There exists a constant $C_2$ that is independent of the mesh size and the location of the interface such that
for each interface edge $F \in \cE^{int}$  the following efficiency bound holds:
	\beq \label{local-efficiency-a}
	\begin{split}
		&\|j_{n,F}\|_{0,F} + \|j_{t,F}\|_{0,F}  \\
		\le& C_2 \!\!\!\! \sum_{K \in \{K_{F,1}, K_{F,2}\}}
		 \!\!\!\!\left( \| \a^{1/2} (\nabla u- \nabla_h u_\cT) \|_{0,K} \!+\!
		 \| \tilde \a^{1/2} \nabla u_\cT\|_{0,S_K}
		\!+\!  H_{f,K} +h_F^{1/2}\min \left\{ \sqrt{h_{F}^-}, \sqrt{h_{F}^+} \right\}
		 \right).
	\end{split}
	\eeq
\end{lemma}

\begin{proof}
Without loss of generality assume that $h_F^+ \ge h_F^-$. Let $K \in \{K_{F,1}, K_{F,2} \}$. Define 
 $\hat K \subset K$ be the triangle such that (i) $F^+$ is a complete edge of $\hat K$ and (ii) the vertex in $K$ that is opposite to $F$, denote by $z_{F,K}$, is also a vertex of $\hat K$. It is obvious that $\hat K$ is regular.
Also define $\hat u_K$ be a piecewise linear function such that $\hat u_K \equiv u_\cT$ on  $K^+$ (or $\hat K$), if $\hat K \subset K^+$ (or $K^+ \subset \hat K$) and that $\hat u_K|_F$ is linear.  It is obvious that $\hat u_K$ uniquely exists.
Finally define $\hat u_{\cT}$ such that $\hat u_{\cT}|_K = \hat u_K$, $K \in \{K_{F,1}, K_{F,2}\}$.
By the triangle inequality, we have
\begin{equation} \label{part0}
\begin{split}
	\| \jump{\tilde \a^{1/2} D_n u_\cT}\|_{0,F} &\le
	\| \jump{\tilde\a^{1/2} D_n \hat u_\cT}\|_{0,F}  + \| \jump{\tilde \a^{1/2} D_n (u_\cT - \hat u_\cT)}\|_{0,F}\\
	& \le 
	\| \jump{\tilde\a^{1/2} D_n \hat u_\cT}\|_{0,F}  +
	\sum_{K \in \{K_{F,1}, K_{F,2}\}}\|\tilde \a^{1/2} \nabla (u_\cT - \hat u_\cT)|_K\|_{0,F}.
\end{split}
\end{equation}
Then by the definition of $\hat u_\cT$ and a direct computation, we have for each $K \in \{K_{F,1}, K_{F,2}\}$:
\begin{equation}
\begin{split}
	&\|\tilde \a^{1/2} \nabla (u_\cT - \hat u_\cT)|_K\|_{0,F} 
	\lesssim  \dfrac{\sqrt{h_{F}^-}}{ \rho_{K^-}} |(u_\cT - \hat u_\cT)(z_F^-)|,
\end{split}
\end{equation}
where $\rho_{K^-}$ denotes that radius of the ball inscribed in $K^-$. This, combining 
with the fact that $|(u_\cT - \hat u_\cT)(z_F^-)| \lesssim \rho_{K^-}$ (See \cite{2003LiLinWu,LiLinLinRo:04} for the consistence of FE and IFE functions when one piece of the interface element is small.) yields
\begin{equation}\label{part1}
\begin{split}
	&\|\a^{1/2} \nabla (u_\cT - \hat u_\cT)|_K\|_{0,F} 
	\lesssim \sqrt{h_{F}^-}.	
\end{split}
\end{equation}
 Now applying the fact that $\jump{D_n \hat u_\cT}|_F$ is a constant and the standard efficiency result on $\hat K_{F,1}$ and $\hat K_{F,2}$ gives
\begin{equation}\label{part2}
\begin{split}
	&h_F^{1/2}\| \jump{\tilde\a^{1/2} D_n \hat u_\cT}\|_{0,F}  \lesssim
	h_F^{1/2}\| \jump{\tilde\a^{1/2} D_n  u_\cT}\|_{0,F^+} \\
	\lesssim &
	\| \a^{1/2} \nabla u - \tilde \a^{1/2} \nabla u_\cT)\|_{\hat K_{F,1} \cup \hat K_{F,2}}
	+ h_F \|f\|_{{\hat K_{F,1} \cup \hat K_{F,2}}}\\
	\lesssim & \sum_{K \in \{K_{F,1}, K_{F,2}\}}\| \a^{1/2} \nabla (u - u_\cT)\|_{K} + \| \tilde \a^{1/2} \nabla u_\cT\|_{0,S_K}
	+  H_{f,K}.
\end{split}
\end{equation}
Combing \eqref{part0}--\eqref{part2} we have
\begin{equation}\label{part3}
\begin{split}
	&h_F^{1/2}\| \jump{\tilde\a^{1/2} D_n u_\cT}\|_{0,F} \\
	\le & C_2 \sum_{K \in \{K_{F,1}, K_{F,2}\}} \left( \| \a^{1/2} \nabla (u - u_\cT)\|_{ K} + h_F^{1/2}\sqrt{h_{F}^-} +H_{f,K} 
	+\| \tilde \a^{1/2} \nabla u_\cT\|_{0,S_K}\right)
\end{split}
\end{equation}
where the constant $C_2$  is independent of the mesh size and the location of the interface.

Similarly we can prove that
\beq\label{part4}
	h_F^{1/2} \|\tilde \a_F^{1/2}j_{t,F}\|_{0,F} \le C_2\sum_{K \in \{ K_{F,1} , K_{F,2} \}}  
	\left( \| \a^{1/2} \nabla (u-u_\cT)\|_{0,K} + h_F^{1/2}\sqrt{h_{F}^-} \right), 
\eeq
where the constant $C_2$  is also independent of the mesh size and the location of the interface.
Finally \eqref{local-efficiency-a} is a direct result of  \eqref{part3} and \eqref{part4}.
\end{proof}

\begin{remark}
\eqref{part3} indicates that for the case when $\lambda$ and $\Lambda$ are extreme the term $h_F^{1/2}\sqrt{h_{F}^-}$ becomes negligible, and, therefore, $C_2$ can be used as an effective efficiency constant.
\end{remark}

\begin{theorem}
Let $u$ and $u_\cT$ be the solution in \eqref{continuous-weak-solution} and \eqref{IFE approx}, respectively. 
The following efficiency bound holds for any $K \in \cT^{int}$:
	\begin{equation}\label{part5}
	\begin{split}
	&\eta_K \le \min\{ C_1, C_2\} 
	\left( \| \a^{1/2} \nabla (u - u_\cT)\|_{0,\omega_K} 
	+\delta_K\right).
	\end{split}
	\end{equation}
\end{theorem}
where 
\[
\delta_K = \sum_{K \subset \o_K} \| \tilde \a^{1/2} \nabla u_\cT\|_{0,S_K}
	+ \chi(C_2\le C_1) h_F^{1/2}\sqrt{h_{F}^-} +\sum_{K \subset \o_K} H_{f,K},
\]
$\chi$ is the characteristic function and
$C_1$ and $C_2$ are the efficiency constants in  Lemma \ref{lem:efficiency1} and Lemma \ref{lem:efficiency2}, respectively.

\section{Numerical Results}\label{sec:6}
In this section, we report some numerical results to demonstrate the performance of the residual-based error estimator for partially penalized IFEM. 

For the first three examples, we consider a diffusion interface problem with a smooth elliptical interface curve which has been reported in 
\cite{LinLinZhang:15, LinYangZhang:15:2}. Let $\Omega=[-1,1]^2$, and the interface $\Gamma$ is an ellipse centered at $(x_0,y_0) = (0,0)$ with horizontal semi-axis $a= \frac{\pi}{6.28}$ and the vertical semi-axis $b = \frac{3}{2}a$. 
The interface separates $\Omega$ into two sub-domains, denoted by $\Omega^-$ and $\Omega^+$ such that
\begin{equation*}
    \Omega^- = \{(x,y): r(x,y) < 1\} ~~~\text{and}~~~
    \Omega^+ = \{(x,y): r(x,y) > 1\},
\end{equation*}
where
\begin{equation*}
  r(x,y) = \sqrt{\frac{(x-x_0)^2}{a^2} + \frac{(y-y_0)^2}{b^2}}.
\end{equation*}
The exact solution to this interface problem is 
\begin{equation}\label{eq: true solution ellipse}
    u(x,y) =
    \left\{
      \begin{array}{ll}
        \frac{1}{\beta^-}r^{p}, & \text{if~} (x,y) \in\Omega^-, \\
        \frac{1}{\beta^+}r^{p} + \frac{1}{\beta^-} - \frac{1}{\beta^+}, &\text{if~} (x,y)\in\Omega^+.
      \end{array}
    \right.
\end{equation}
Here $\beta^\pm>0$ are the diffusion coefficients, and $p>0$ is the regularity parameter. In the following, we  use $\rho=\frac{\beta^+}{\beta^-}$ to denote
 the ratio of the coefficient jump.

Our adaptive mesh refinement  
follows the standard procedure:
\begin{equation*}
\textbf{Solve} \longrightarrow ~ \textbf{Estimate} \longrightarrow~ \textbf{Mark} \longrightarrow~ \textbf{Refine}.
\end{equation*}
We solve the interface problem using IFEM in \eqref{IFE approx},  then we compute the residual-based error indicator $\eta_K$ on each element by \eqref{eta-K}. We adopt the equilibration marking strategy, i.e., construct a minimal subset $\hat{\mathcal{T}}$ of $\mathcal{T}$ such that 
\begin{equation}
	\sum_{K\in\hat{\mathcal{T}}} \eta_K^2 \geq \theta^2\eta^2
\end{equation}
where the threshold $\theta=0.5$. Finally we refine the marked triangles by newest vertex bisection \cite{Maubach:95}.
The initial mesh is formed by first partitioning the domain into a $4\times 4$ congruent rectangles, then cutting each rectangle into two congruent triangles by connecting its diagonal with positive slope.

\subsection*{Example 6.1 (Piecewise smooth solution with moderate jump)}
\label{ex1} 

In this example, we let $\rho = 100$ and $p=5$ which represents a moderate coefficient contrast and piecewise smooth solution.
In Figure \ref{fig: mesh small}, we list, from left to right, some typical meshes of similar number of elements and degrees of freedom (DOF) generated by the uniform IFEM, the adaptive IFEM, and the adaptive FEM on unfitted mesh\cite{ChXiZh:09}.  We observe that there is not much local mesh refinement around the interface for the adaptive IFEM in the middle of Figure \ref{fig: mesh small}, comparing with the uniform mesh in the left plot. This indicates that the errors of the partially penalized IFEM on uniform mesh are almost equally distributed, and IFEM itself can resolve the interface accurately for moderate coefficient jump. Whereas using the finite element method on unfitted meshes requires much more local mesh refinement around the interface (see the plot on the right of Figure \ref{fig: mesh small}) to resolve the non-smooth feature of interface problems. 

\begin{figure}[ht!]
\begin{center}
\includegraphics[width=0.32\textwidth]{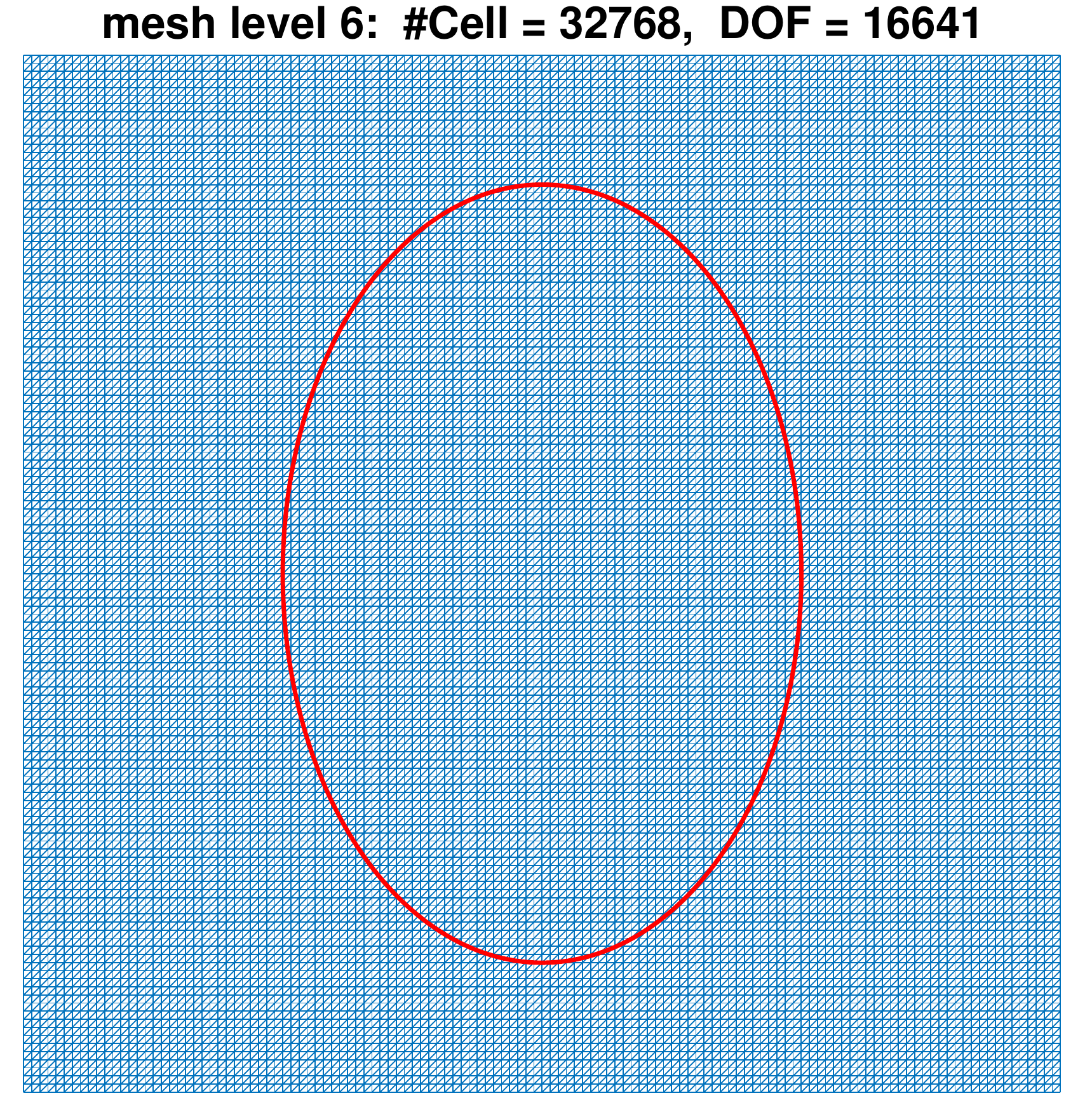}
\includegraphics[width=0.32\textwidth]{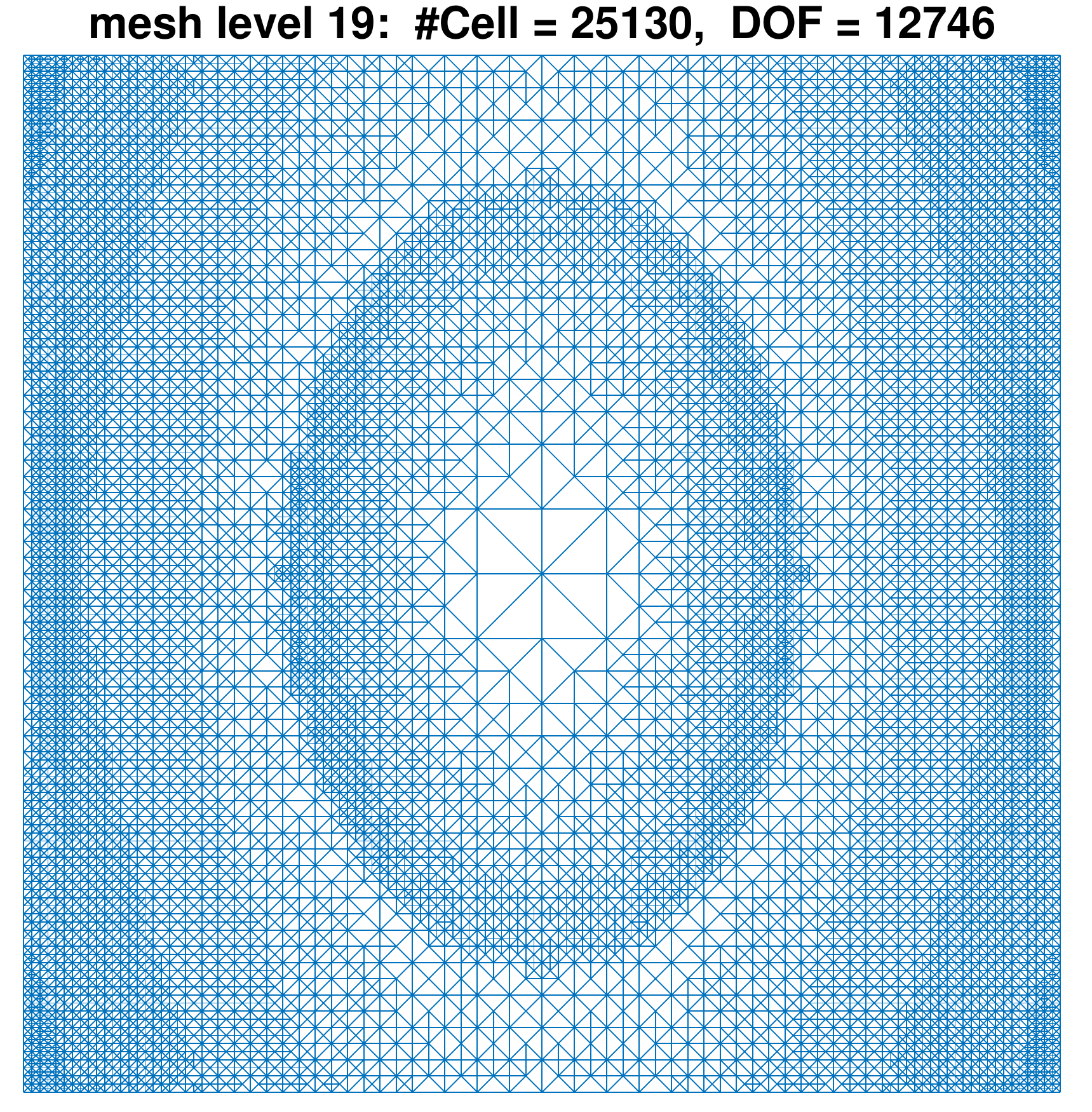}
\includegraphics[width=0.32\textwidth]{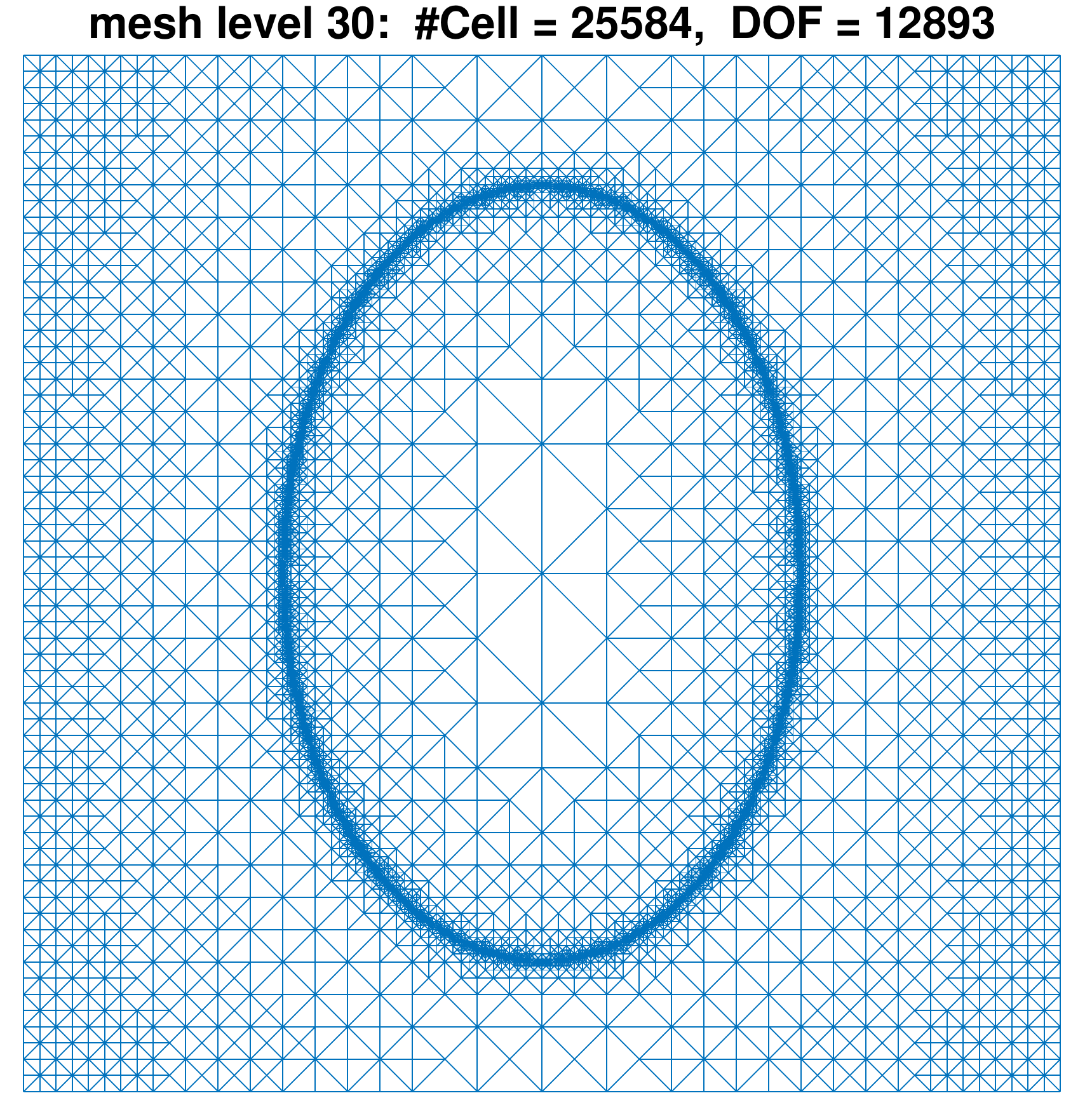}
\end{center}
\caption{Uniform mesh for IFEM (left), adaptive mesh for IFEM (center), adaptive mesh for FEM {\em\cite{ChXiZh:09}} (right) for Example 6.1}
\label{fig: mesh small}
\end{figure}

\begin{figure}[ht!]
\begin{center}
\includegraphics[width=0.6\textwidth]{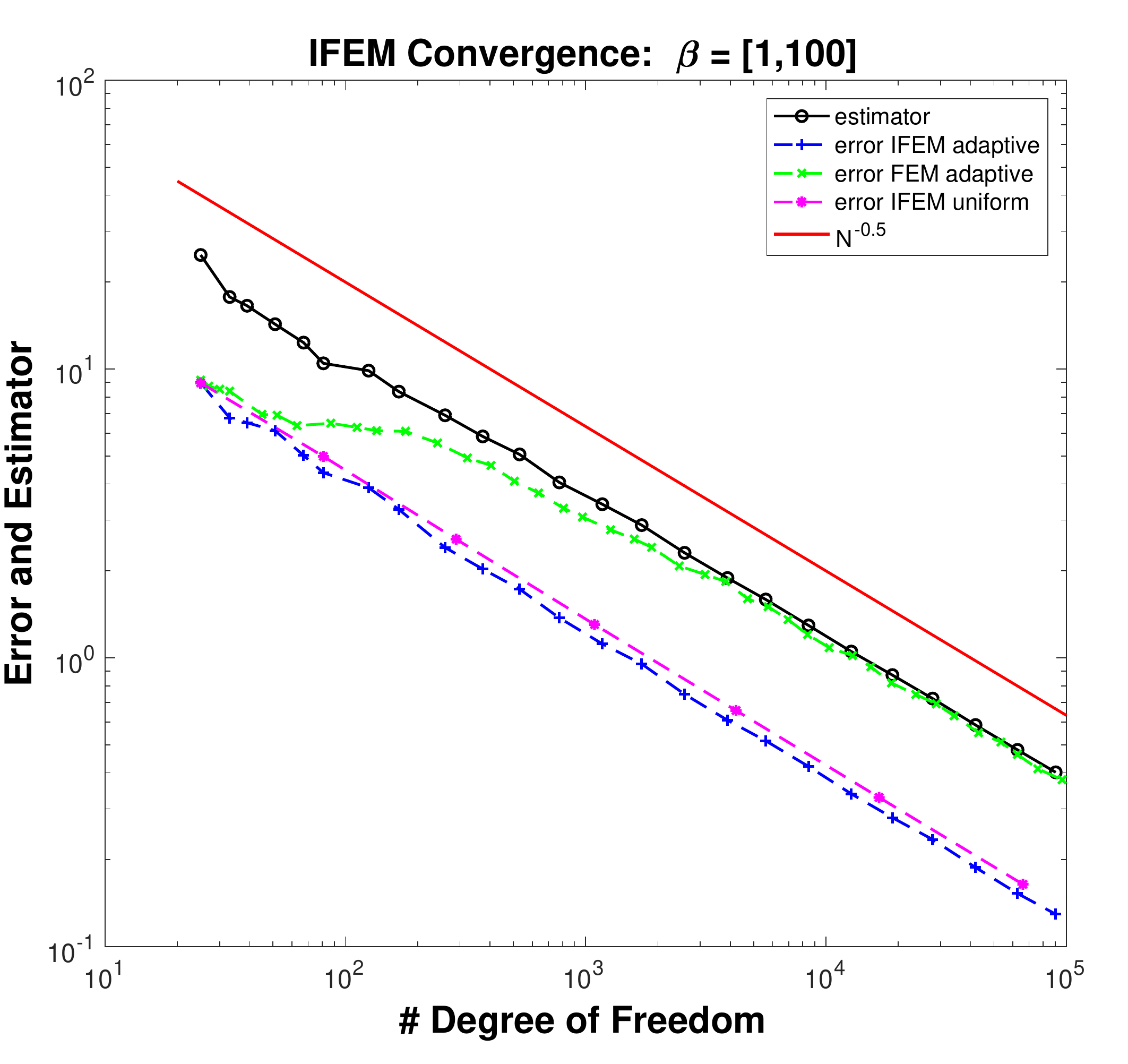}
\end{center}
\caption{Convergence of uniform IFEM, adaptive IFEM, and adaptive FEM for  Example 6.1}
\label{fig: cvg small}
\end{figure}

In Figure \ref{fig: cvg small}, we report the convergence of these three methods 
and the residual-based error estimator for the adaptive IFEM.  
The slopes of the log(DOF)-log$(\| \alpha^{1/2}(\nabla u-\nabla_h u_\cT)\|_{0,\Omega} )$
and the log(DOF)-log($\eta$) for the adaptive IFEM are both very close to $-1/2$, 
which indicates the optimal-order decay of errors with respect to the number of unknowns and, hence, the efficiency of our local error indicators. We use the following efficiency index,
\[
	\mbox{eff-index} = \dfrac{{\eta}}{\| \alpha^{1/2}(\nabla u-\nabla_h u_\cT)\|_{0,\Omega}}
\]
to test the efficiency of our residual-based error estimator. The eff-index is very stable at every mesh level and the value is around 3. We note that the errors of uniform IFEM are very close to the errors of adaptive IFEM. This again indicates that the IFEM itself can resolve the interface accurately for moderate coefficient jumps and piecewise smooth solutions. However, using the standard FEM, the magnitudes of errors are much larger than those of IFEM, with similar degrees of freedom.

\subsection*{Example 6.2 (Piecewise smooth solution with large  jump)}
\label{ex2}
In this example, we test the large jump case by choosing $\rho=10^6$.  In this case, the true solution possesses a very steep gradient at the interface.

The left plot of Figure \ref{fig: IFEM large} shows a typical mesh for the adaptive IFEM, which, compared with Figure \ref{fig: mesh small}, has much denser refinement around the interface. In the right plot of Figure \ref{fig: IFEM large}, we observe the optimal-rate decay of the errors and the estimators.
Nevertheless, even if the convergence rate is optimal for the uniform IFEM, the magnitudes of errors are significantly larger than the errors of adaptive IFEM.
Hence, applying adaptive mesh refinement is computationally more efficient for interface problems with large coefficient jump even for IFEM. 
The efficiency indices are between 2.5 - 3.5 on all meshes except the first few coarse ones, and they become more stable (close to 3) as the computations reach the asymptotical convergence region. This phenomenon indicates the robustness of the error estimation with respect to the ratio of coefficient jump. 

Furthermore, the numerical solutions and the error surfaces on the uniform and adaptive meshes with similar DOFs are depicted in Figure \ref{fig: sol large} and Figure \ref{fig: error large}, respectively. We can observe that the error is significantly diminished for the adaptive solution.

\begin{figure}[ht!]
\begin{center}
\includegraphics[width=0.48\textwidth]{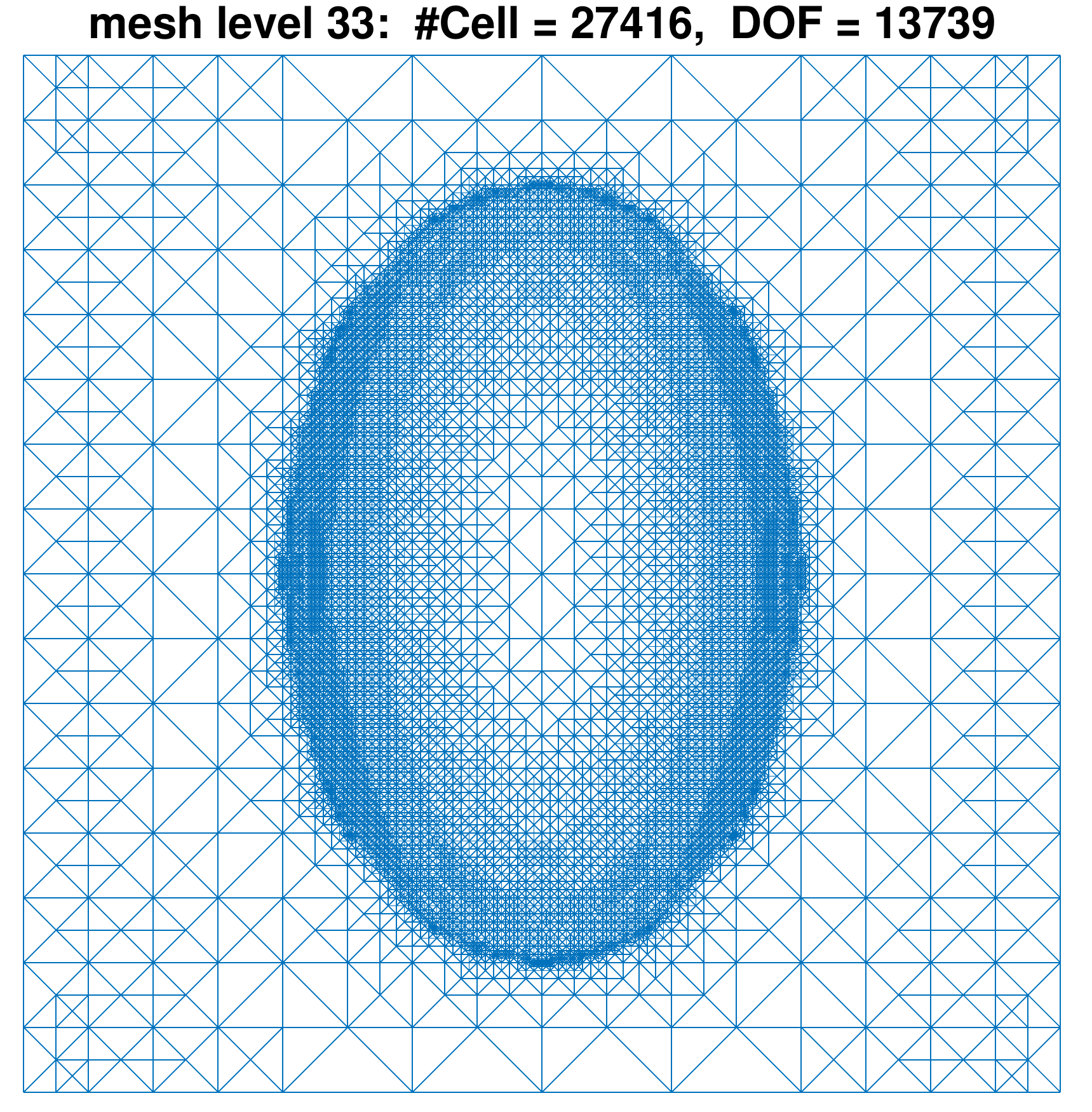}
\includegraphics[width=0.5\textwidth]{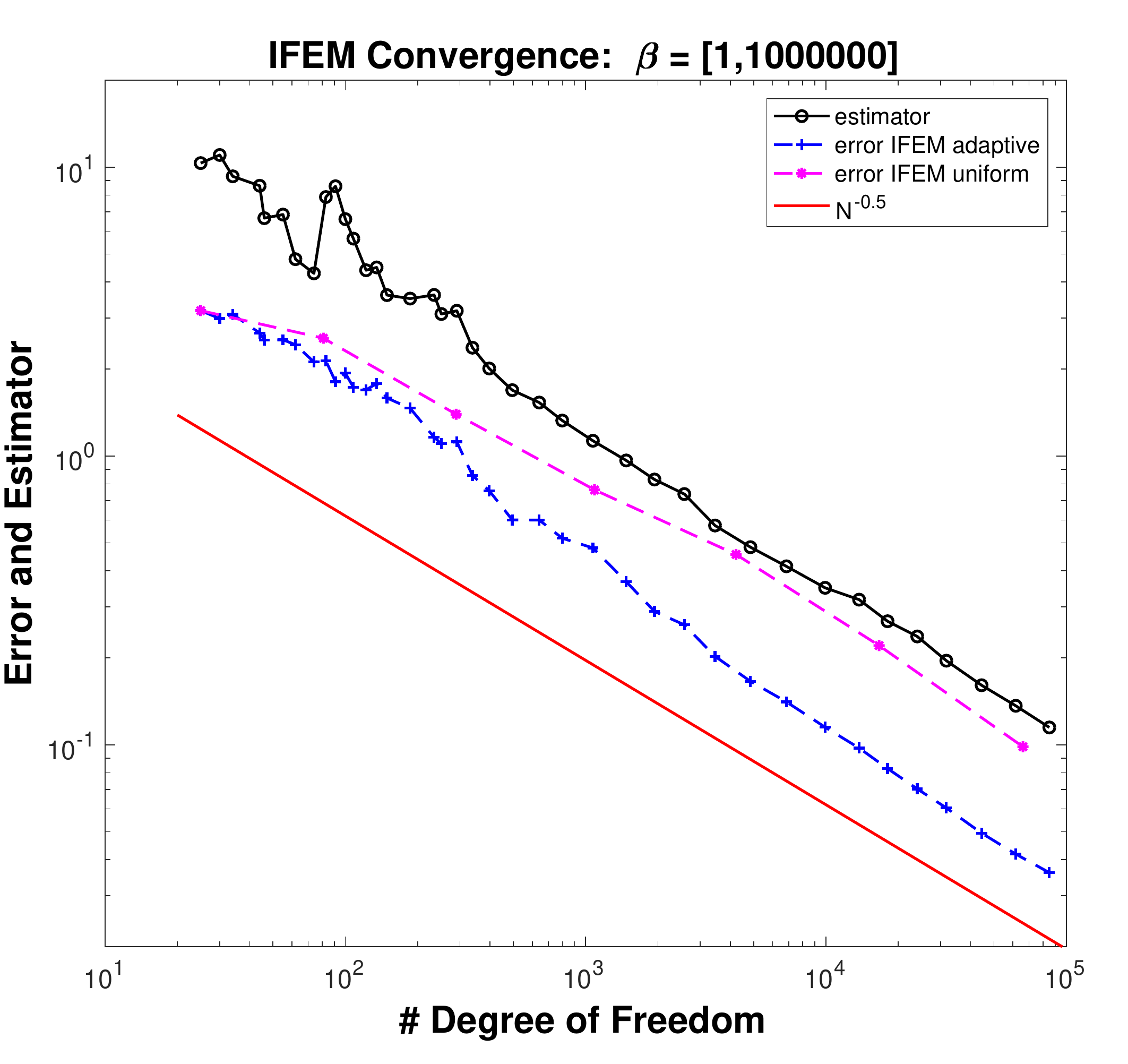}
\end{center}
\caption{Mesh generated by the adaptive IFEM (left) and the convergence of 
adaptive and uniform IFEM (right) for Example 6.2}
\label{fig: IFEM large}
\end{figure}

\begin{figure}[ht!]
\begin{center}
\includegraphics[width=0.5\textwidth]{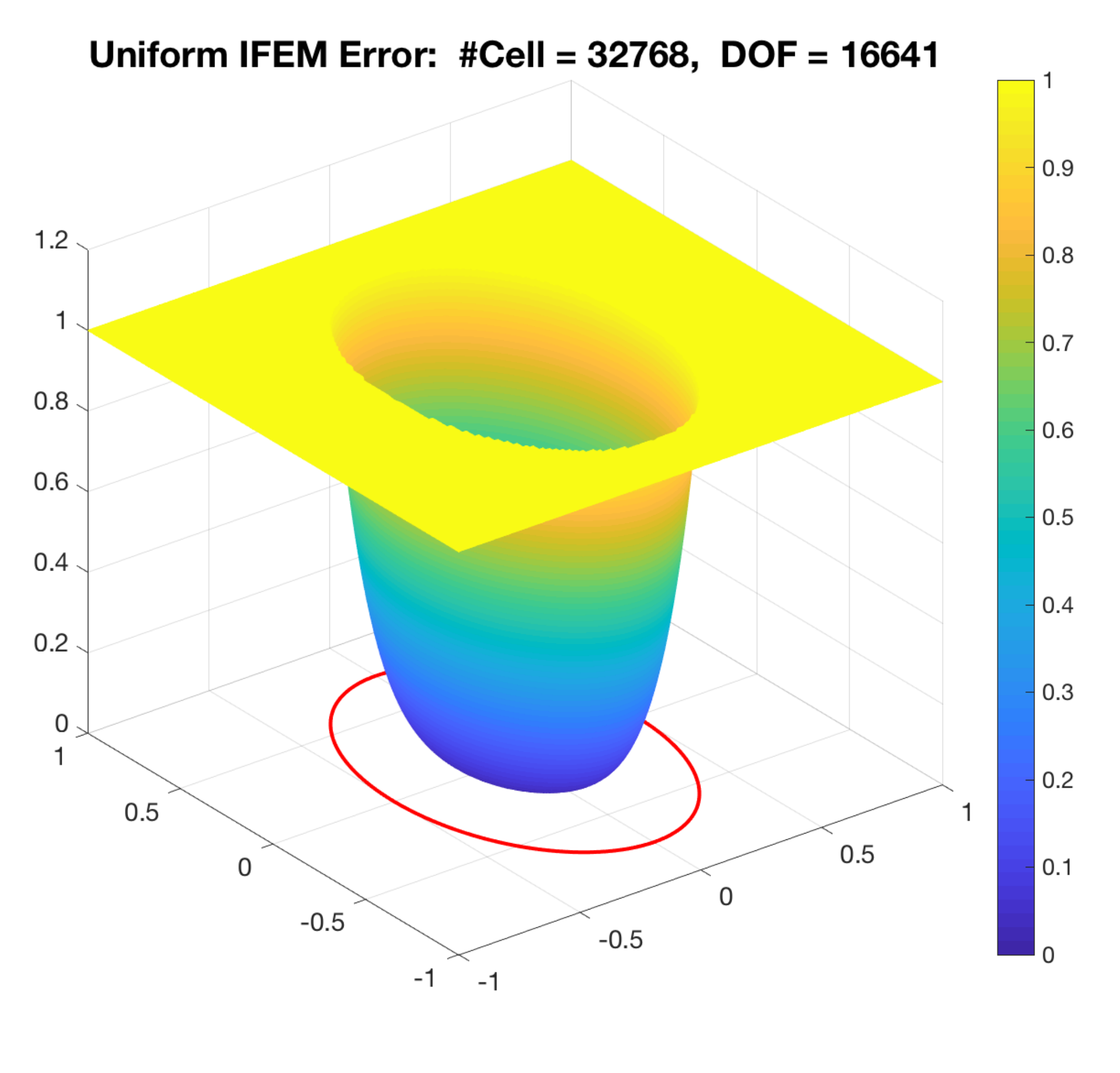}~
\includegraphics[width=0.5\textwidth]{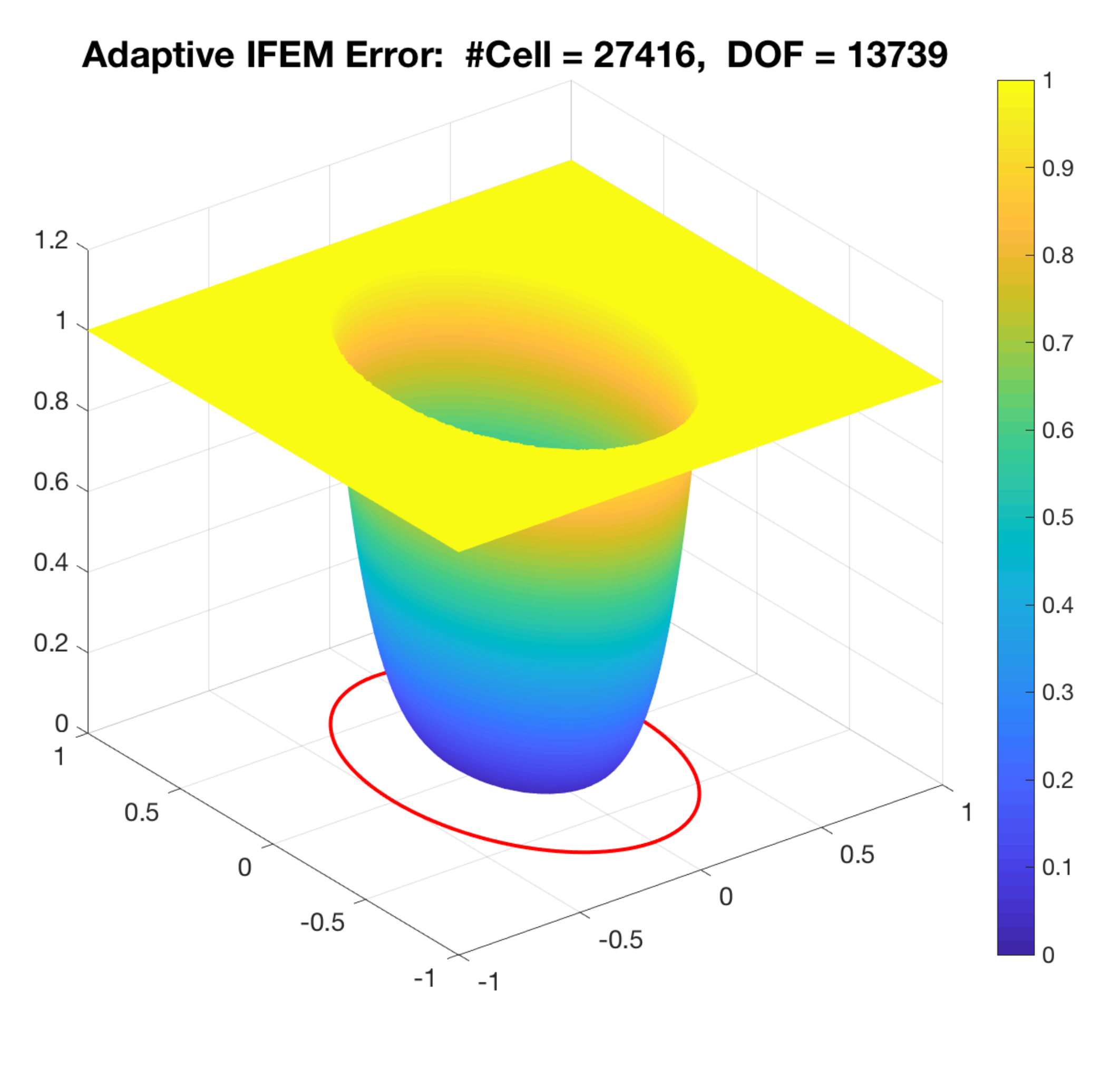}
\end{center}
\caption{Numerical solutions of uniform (left) and adaptive (right) IFEM with similar DOFs for Example 6.2}
\label{fig: sol large}
\end{figure}

\begin{figure}[ht!]
\begin{center}
\includegraphics[width=0.5\textwidth]{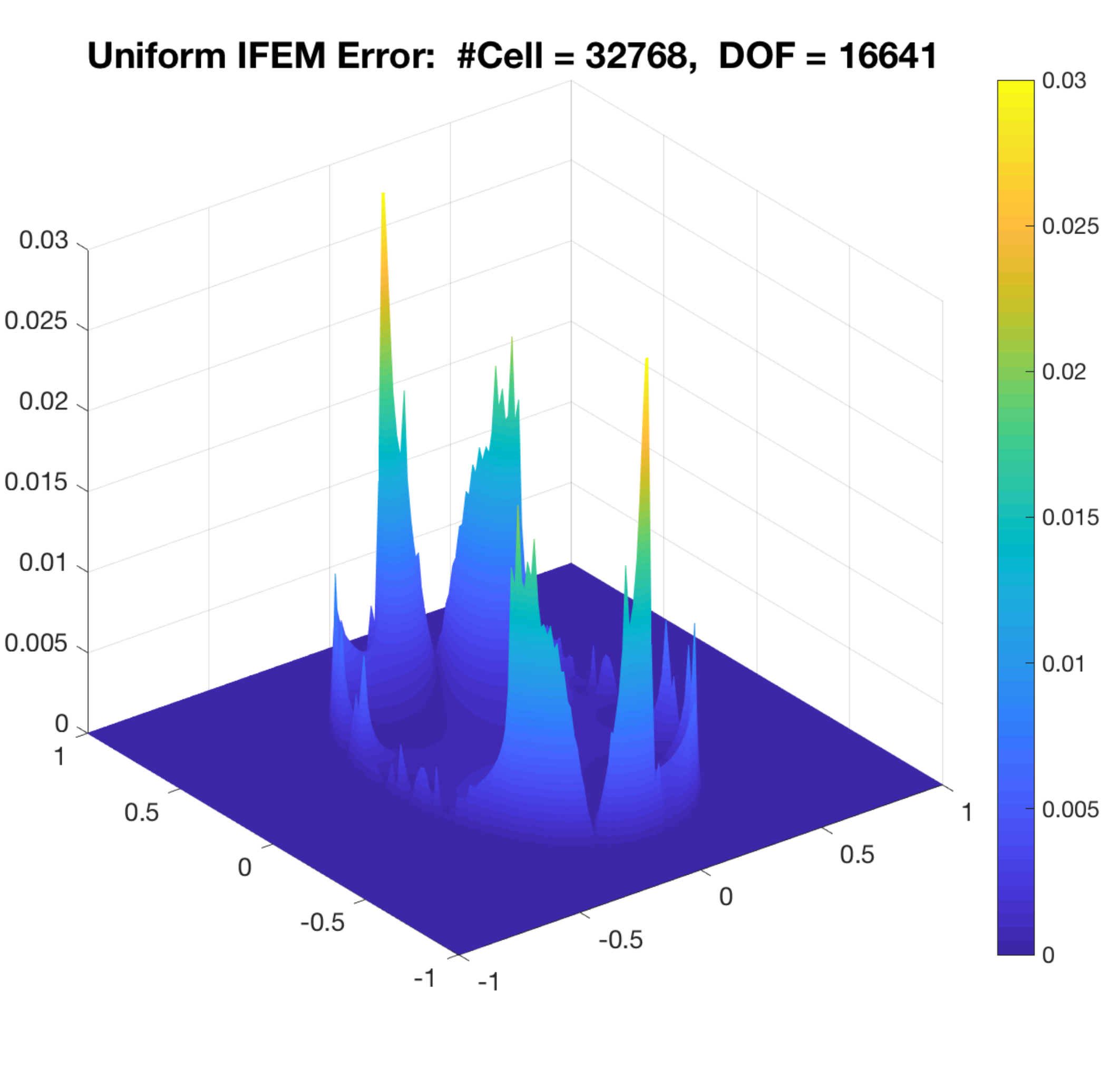}~
\includegraphics[width=0.5\textwidth]{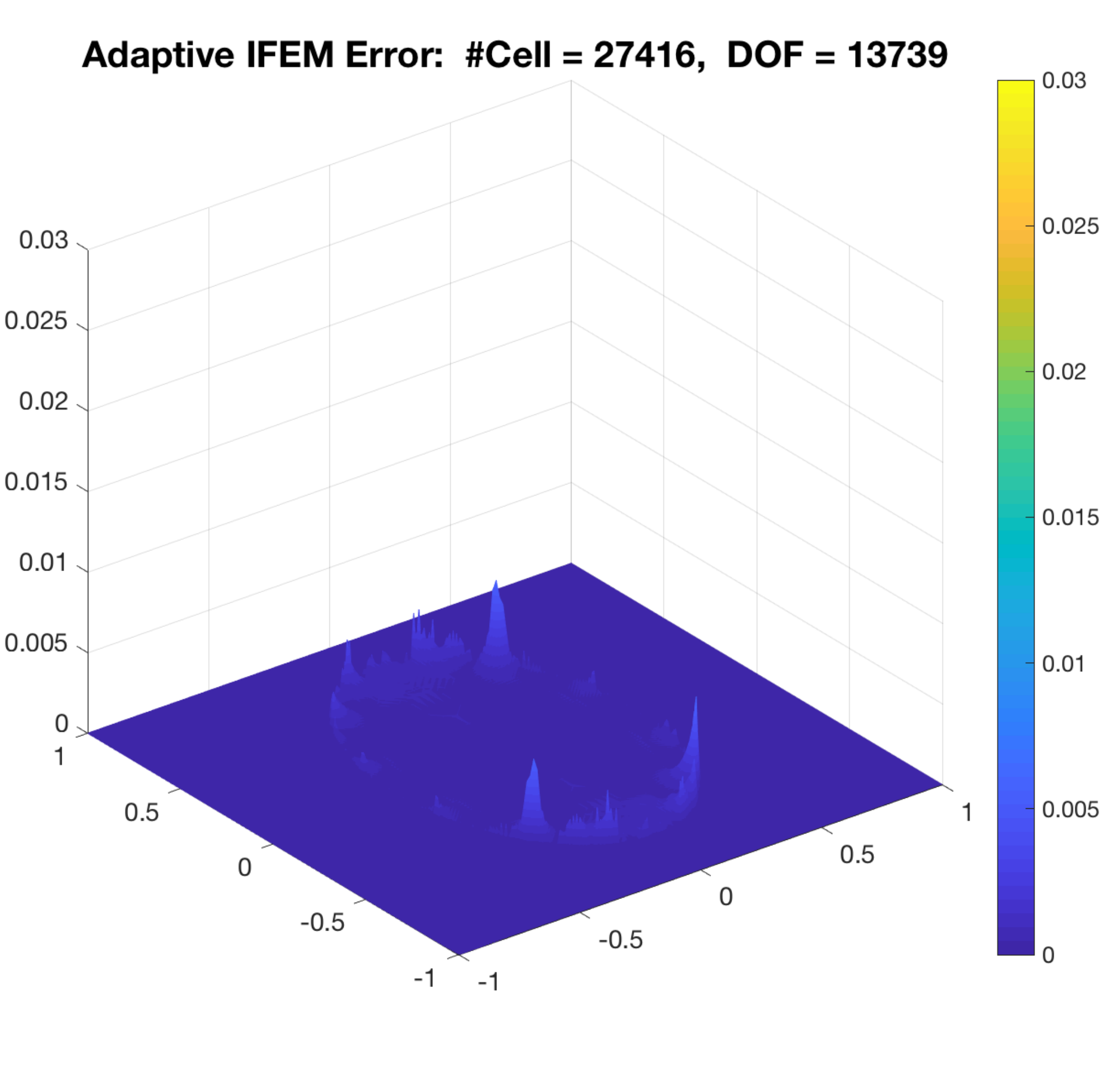}
\end{center}
\caption{Point-wise errors of uniform (left) and adaptive (right) IFEM with similar DOFs  for Example 6.2}
\label{fig: error large}
\end{figure}

\subsection*{Example 6.3 (Solution with singularity and large coefficient jump)}
\label{ex3}

In this example, we consider the interface problem with large coefficient jump and a solution singularity. We choose $\rho=10^6$ and $p=0.5$ in the exact solution Figure \ref{eq: true solution ellipse}. Note that the solution is merely in $H^{1.5-\epsilon}(\Omega)$ for any $\epsilon>0$ and the solution becomes singular at the origin. 
 
The left plot of Figure \ref{fig: singularity} shows a typical mesh of the adaptive IFEM, and
we observe that the mesh is densely refined around the interface as well as the point of singularity. The right plot of Figure \ref{fig: singularity} shows the optimal-rate decay of the errors and the estimators of our adaptive IFEM.  
Again, the averaging effectivity index is close to 3, which is similar to that in previous examples. This indicates the uniform effectivity of the error estimate with respect to the type of elements, i.e., interface and non-interface elements. The comparison of error in the right plot of  Figure \ref{fig: singularity} shows a stronger superiority of the adaptive mesh refinement than the uniform mesh refinement. In fact, the convergence of IFEM with uniform mesh refinement is not optimal, due to the singularity of the solution. 

The numerical solutions and error surfaces for the adaptive IFEM and uniform IFEM are depicted in Figure \ref{fig: sol singularity} and Figure \ref{fig: error singularity}, respectively. It is easy to see that the numerical solution on uniform mesh cannot resolve the behavior of exact solution accurately at the singularity point, and the error of the uniform solution possesses a very high peak around the singular point, while we 
can barely observe this phenomenon from the adaptive solution.  

\begin{figure}[ht!]
\begin{center}
\includegraphics[width=0.48\textwidth]{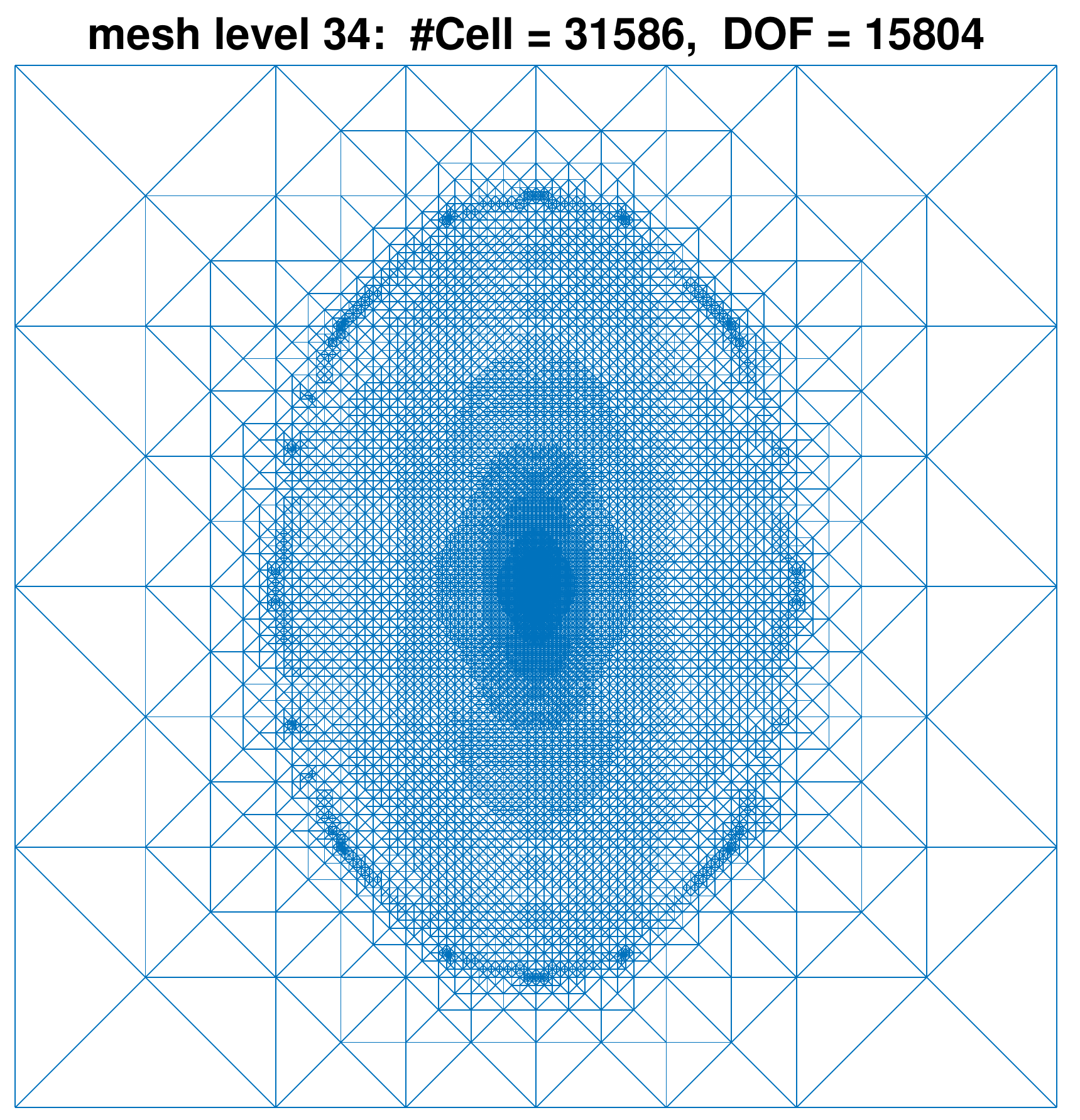}
\includegraphics[width=0.5\textwidth]{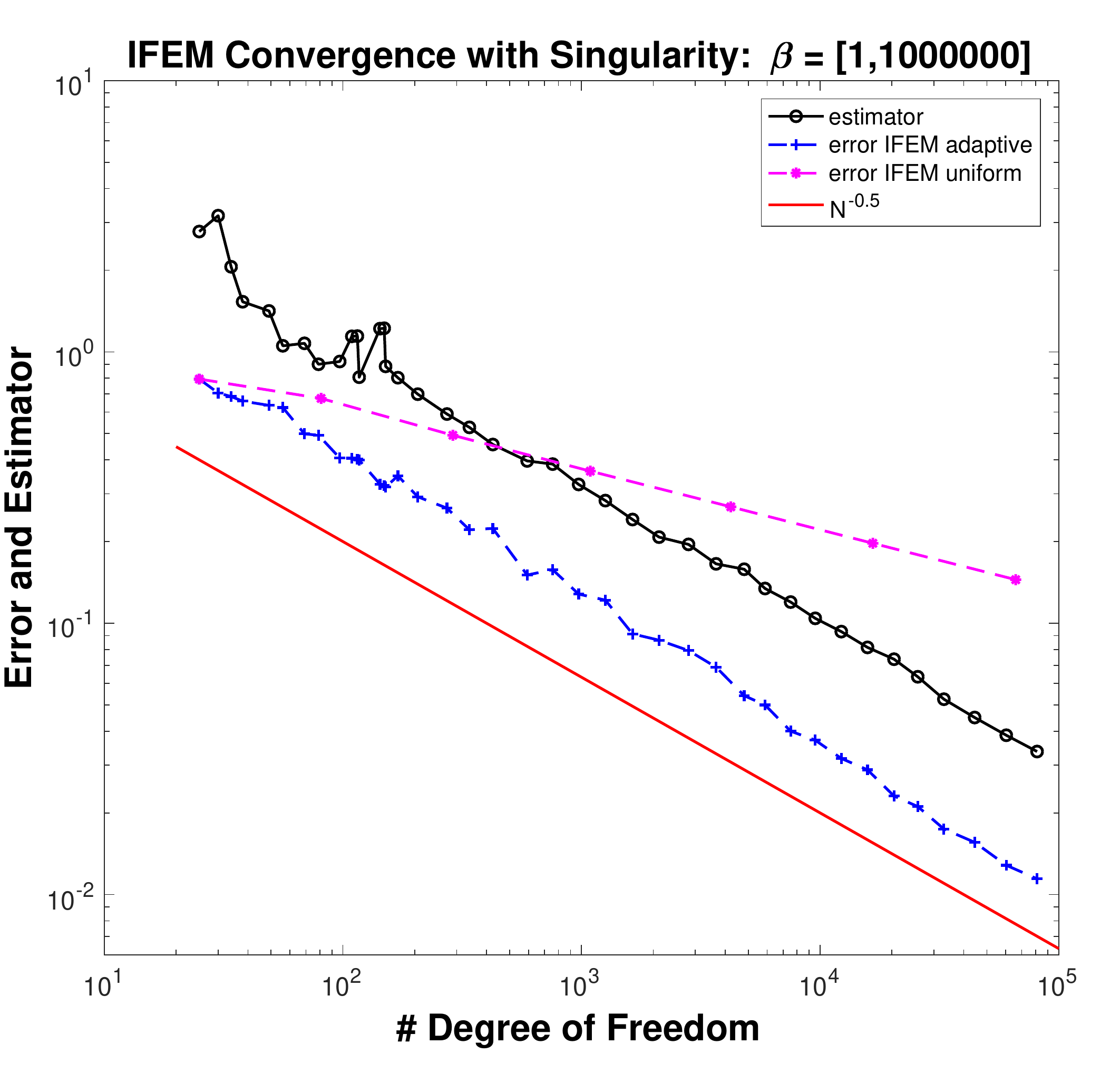}
\end{center}
\caption{Mesh generated by the adaptive IFEM (left) and the convergence of 
adaptive and uniform IFEM (right) for Example 6.3}
\label{fig: singularity}
\end{figure}

\begin{figure}[ht!]
\begin{center}
\includegraphics[width=0.49\textwidth]{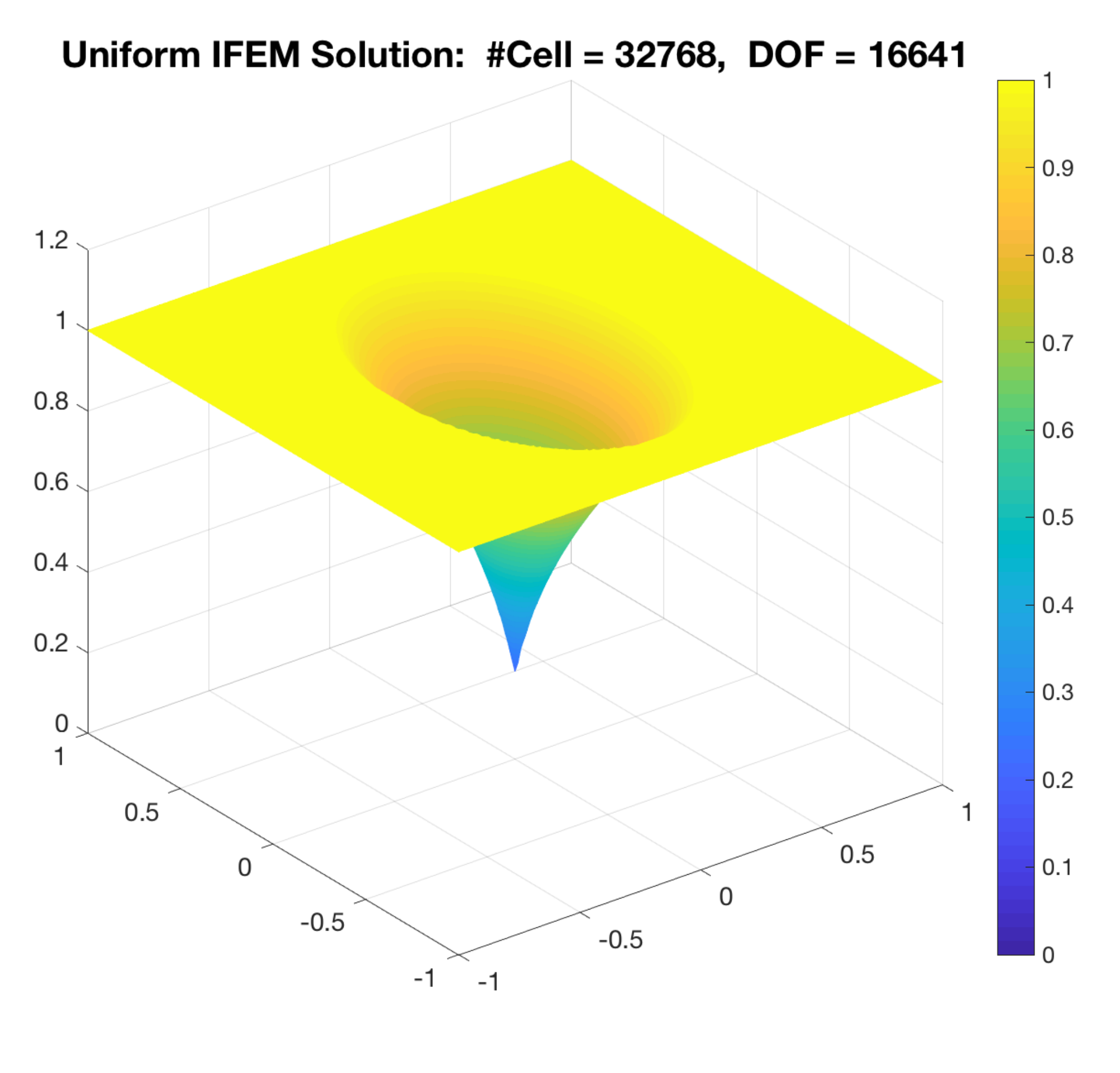}
\includegraphics[width=0.49\textwidth]{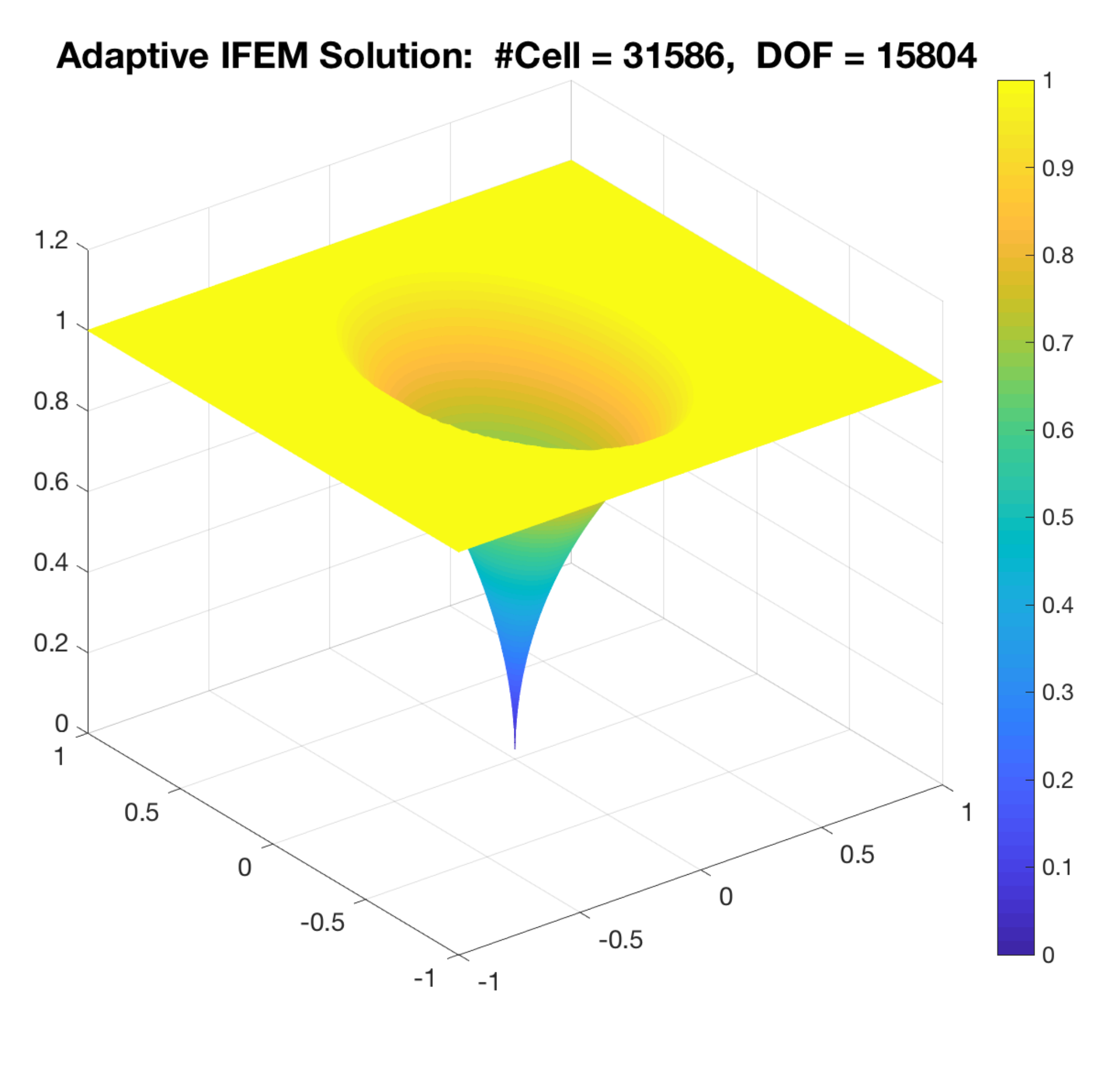}
\end{center}
\caption{Numerical solutions of uniform (left) and adaptive (right) IFEM with similar DOFs for Example 6.3}
\label{fig: sol singularity}
\end{figure}

\begin{figure}[ht!]
\begin{center}
\includegraphics[width=0.49\textwidth]{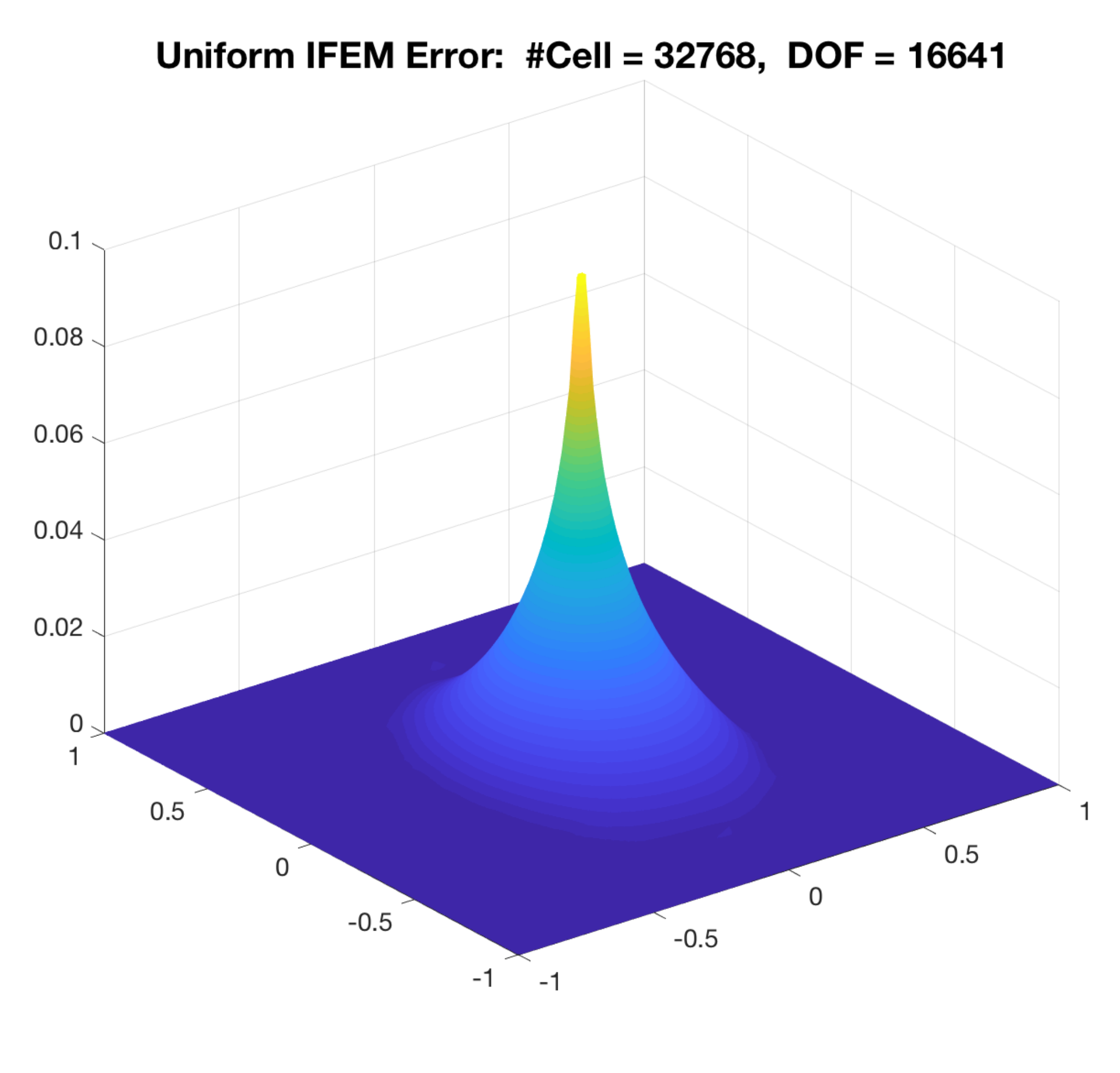}
\includegraphics[width=0.49\textwidth]{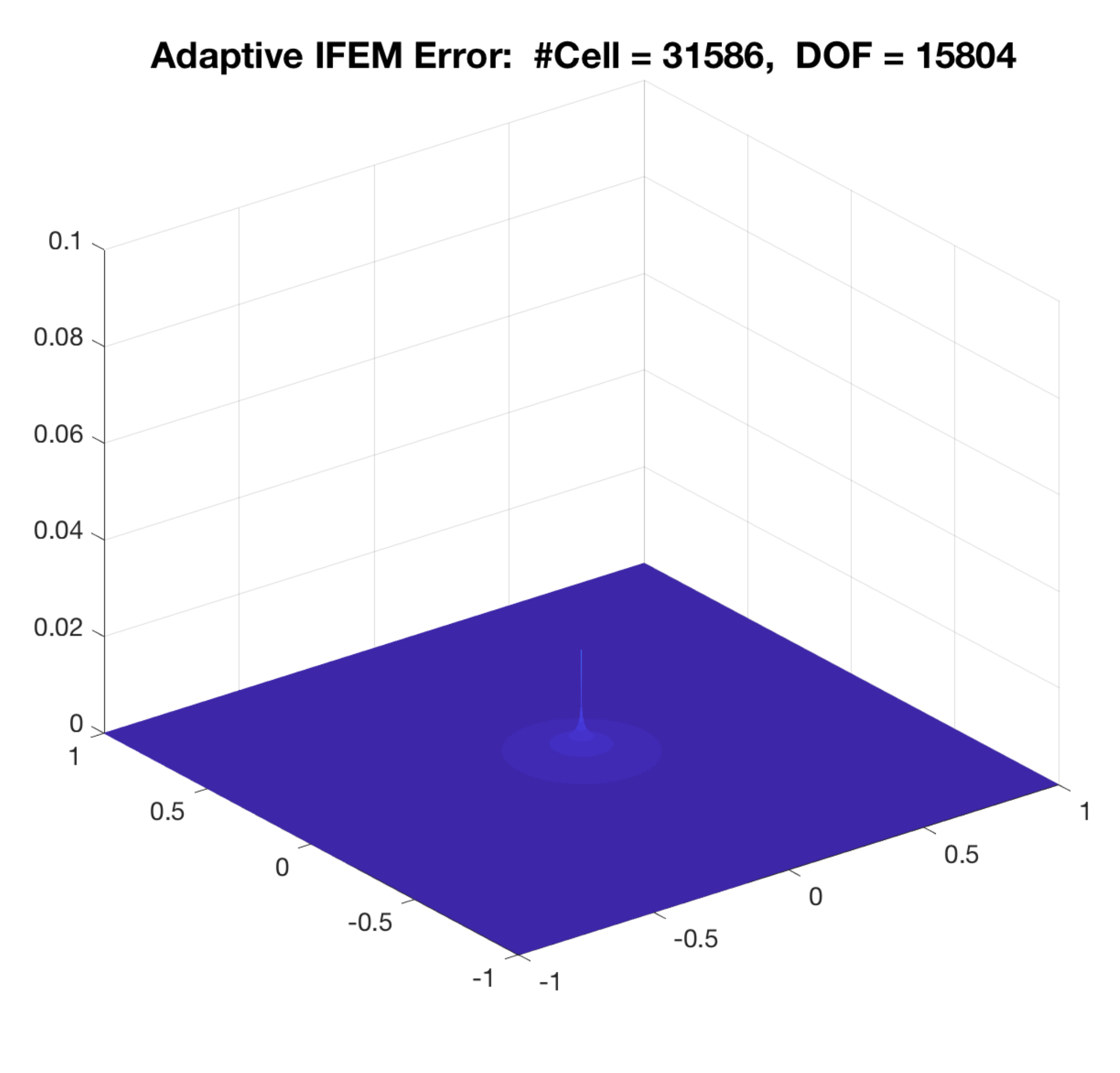}
\end{center}
\caption{Point-wise errors of  uniform (left) and adaptive (right)  IFEM with similar DOFs for Example 6.3}
\label{fig: error singularity}
\end{figure}

\subsection*{Example 6.4 (Solution with complicated interface shape)}
\label{ex4}

In this example, we consider an interface problem with a  
more complicated interfacial shape. The exact solution has a petal-shaped interface and it is defined through the following level set function:
\begin{equation}\label{eq: true solution 2}
    u(x,y) =
    \left\{
      \begin{array}{ll}
        \frac{1}{\beta^-}\phi(x,y), & \text{if~} \phi(x,y) < 0, \\
        \frac{1}{\beta^+}\phi(x,y), &\text{if~} \phi(x,y) \geq 0,
      \end{array}
    \right.
    \quad
\mbox{in } \O = [-1\,,1]^2
\end{equation}
where 
\[
\phi(x,y) = (x^2+y^2)^2\left(1+0.5\sin\left(12\tan^{-1}\left(\frac{y}{x}\right)\right)\right)-0.3 .
\] 
Due to the complexity of interface shape, we start the AMR procedure with a finer initial mesh, a $16\times 16$ Cartesian triangular mesh. A typical mesh is depicted on the left plot of Figure \ref{fig: petal}.
Comparing with Example 6.1, in which the interface is an ellipse (see Figure \ref{fig: mesh small}), the refinement around interface is denser. This is because the larger curvature of the interface causes the larger value of the inconsistency term in the error indicator in \eqref{eta-K}. 

The convergence plot depicted in the right plot of Figure \ref{fig: petal} indicates the optimal-rate decay for both the errors and estimators. The errors of adaptive solution are a little smaller than the errors of the uniform solution with similar degrees of freedom, although the latter also converge in optimal rate. Moreover, the efficient index for this example is close to $3$ which is similar to those in the previous examples. Numerical solutions and error surfaces of uniform and adaptive IFEMs are reported in  Figure \ref{fig: solution petal} and Figure \ref{fig: err petal}, respectively. We can again see that the errors of adaptive IFEM are smaller than the errors of uniform IFEM given similar degrees of freedom. 

\begin{figure}[t]
\begin{center}
\includegraphics[width=0.48\textwidth]{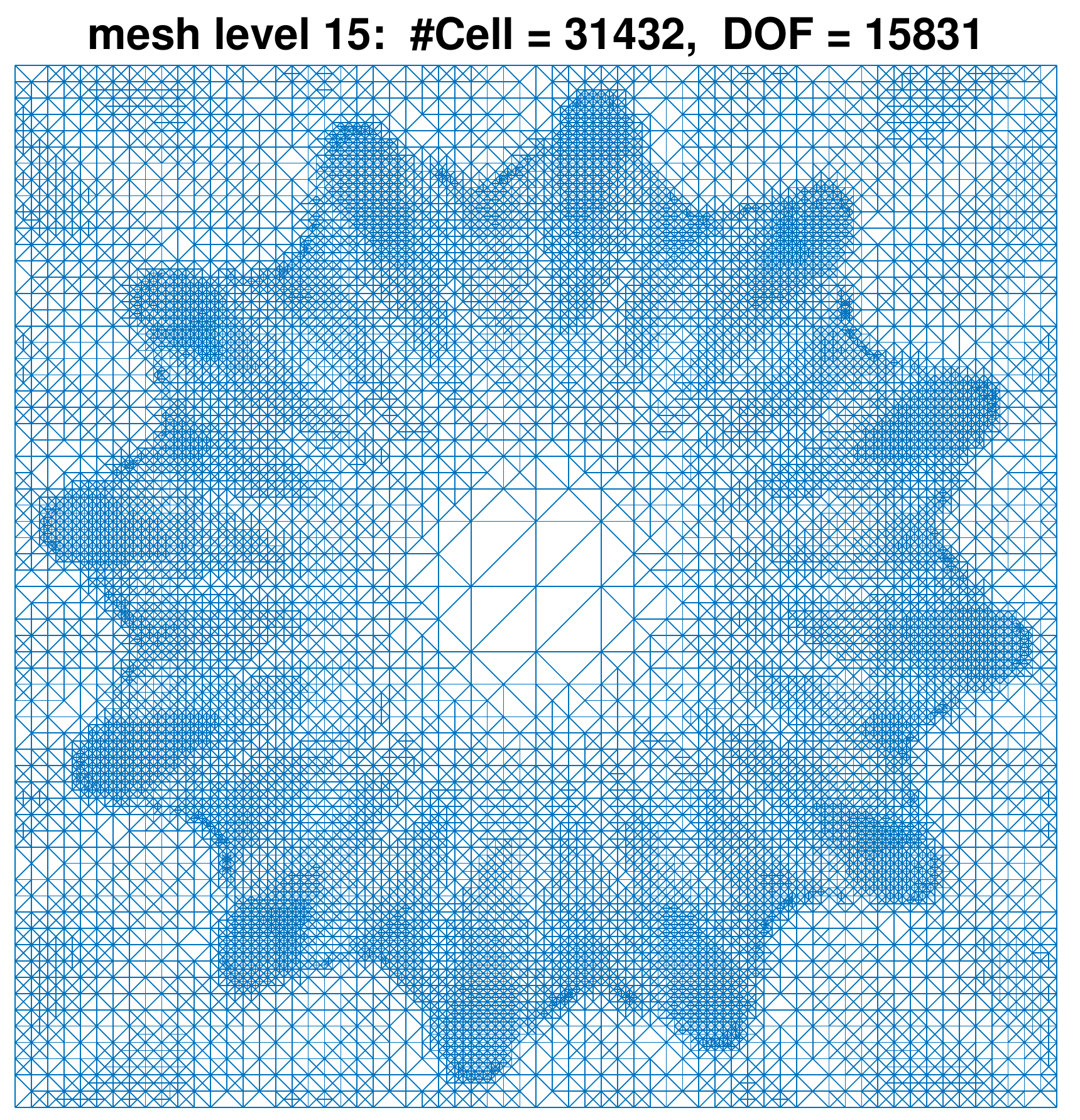}
\includegraphics[width=0.5\textwidth]{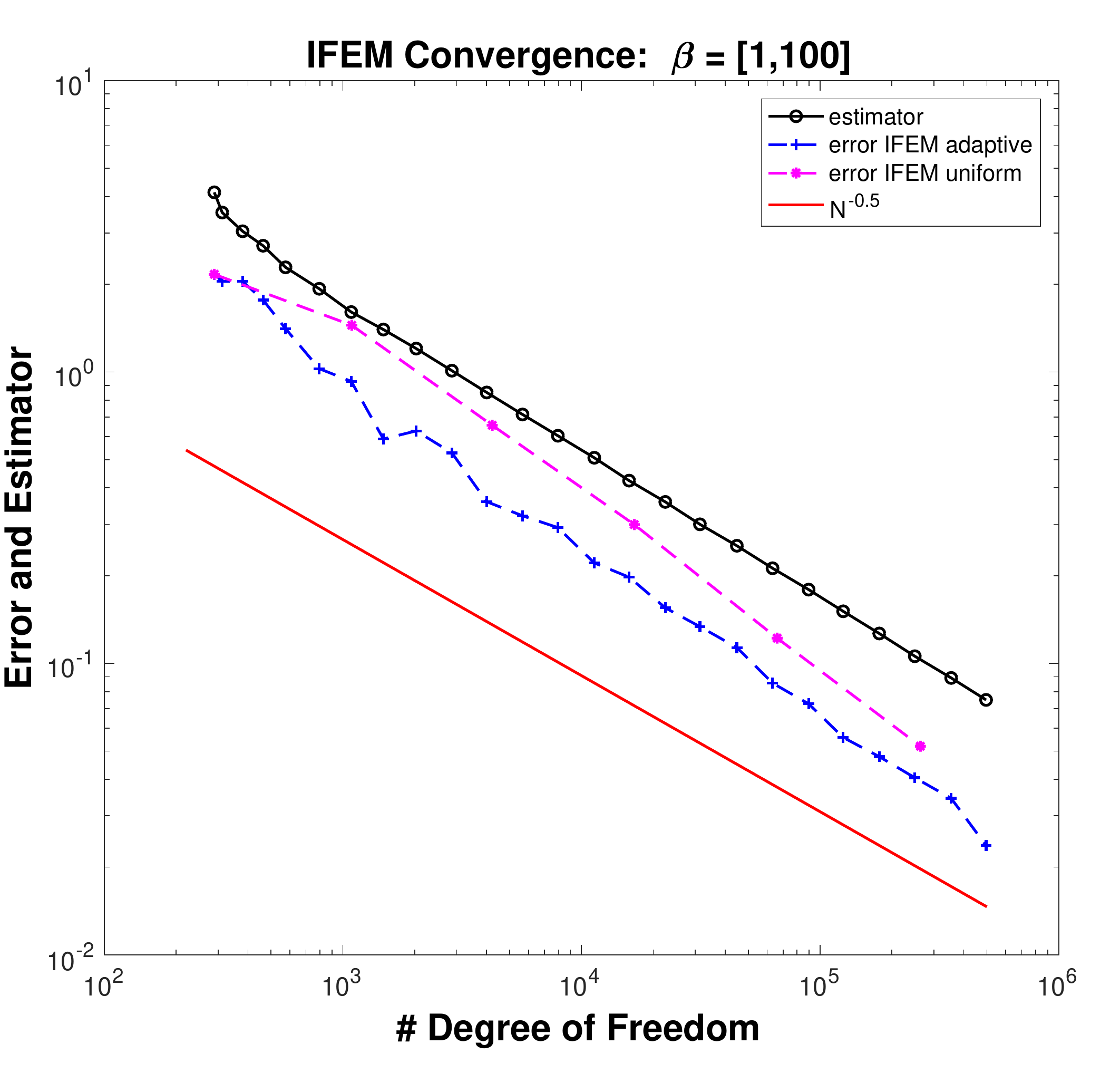}
\end{center}
\caption{Mesh generated by the adaptive IFEM (left) and the convergence of 
adaptive and uniform IFEM for Example 6.4}
\label{fig: petal}
\end{figure}

\begin{figure}[t]
\begin{center}
\includegraphics[width=0.48\textwidth]{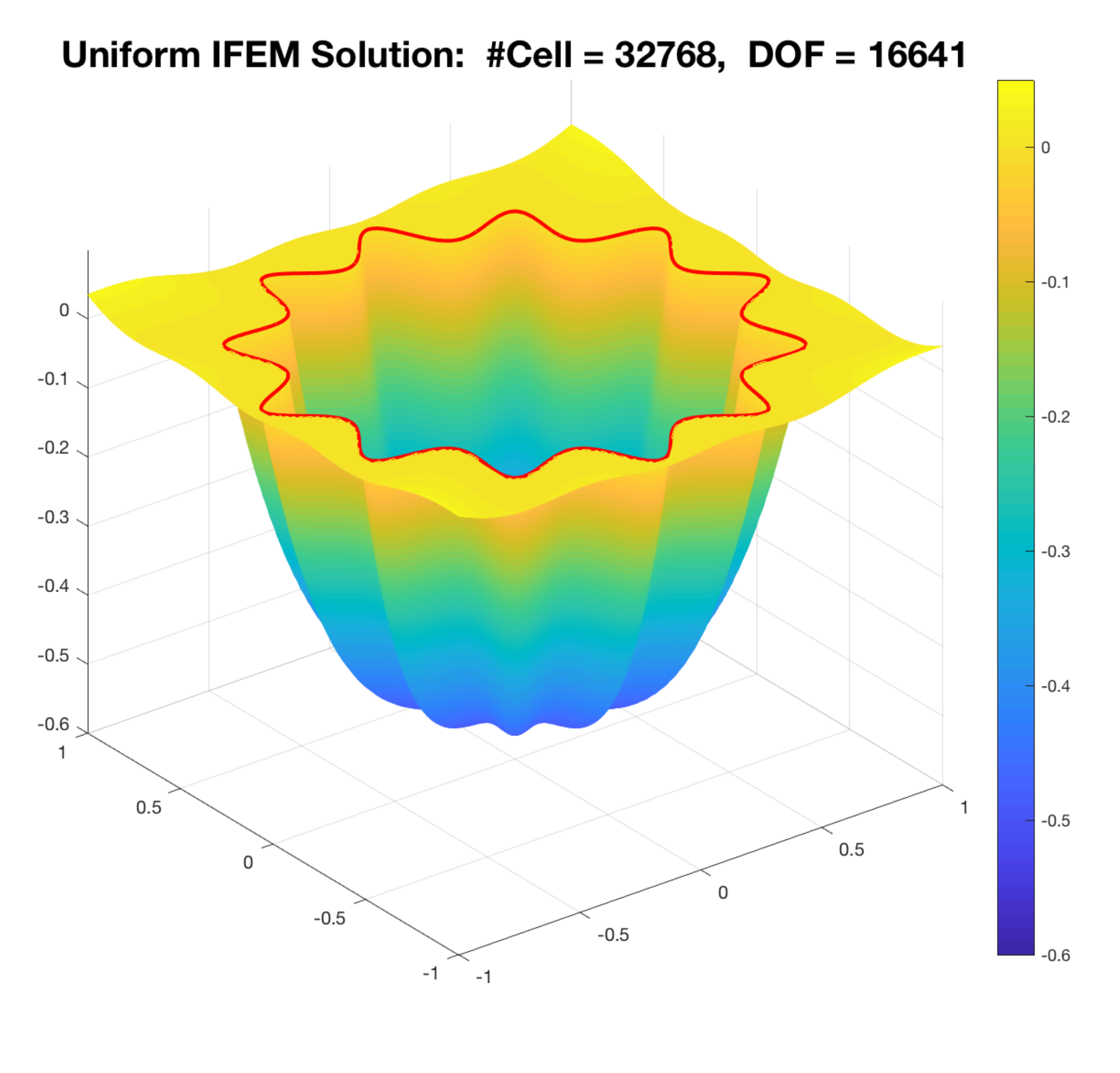}~
\includegraphics[width=0.48\textwidth]{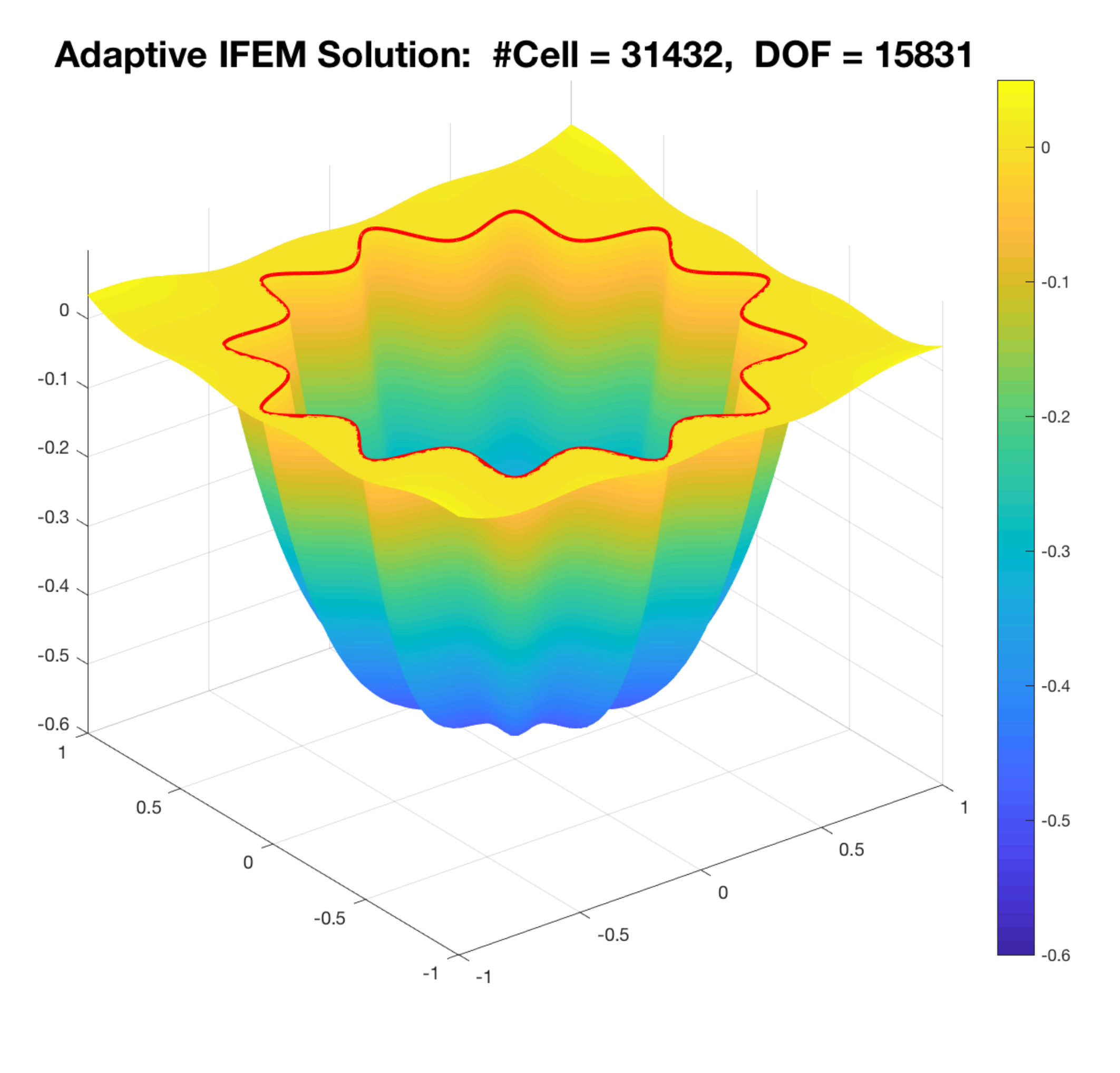}
\end{center}
\caption{Numerical solutions of uniform (left) and adaptive (right) IFEM with similar DOFs for Example 6.4}
\label{fig: solution petal}
\end{figure}

\begin{figure}[t]
\begin{center}
\includegraphics[width=0.48\textwidth]{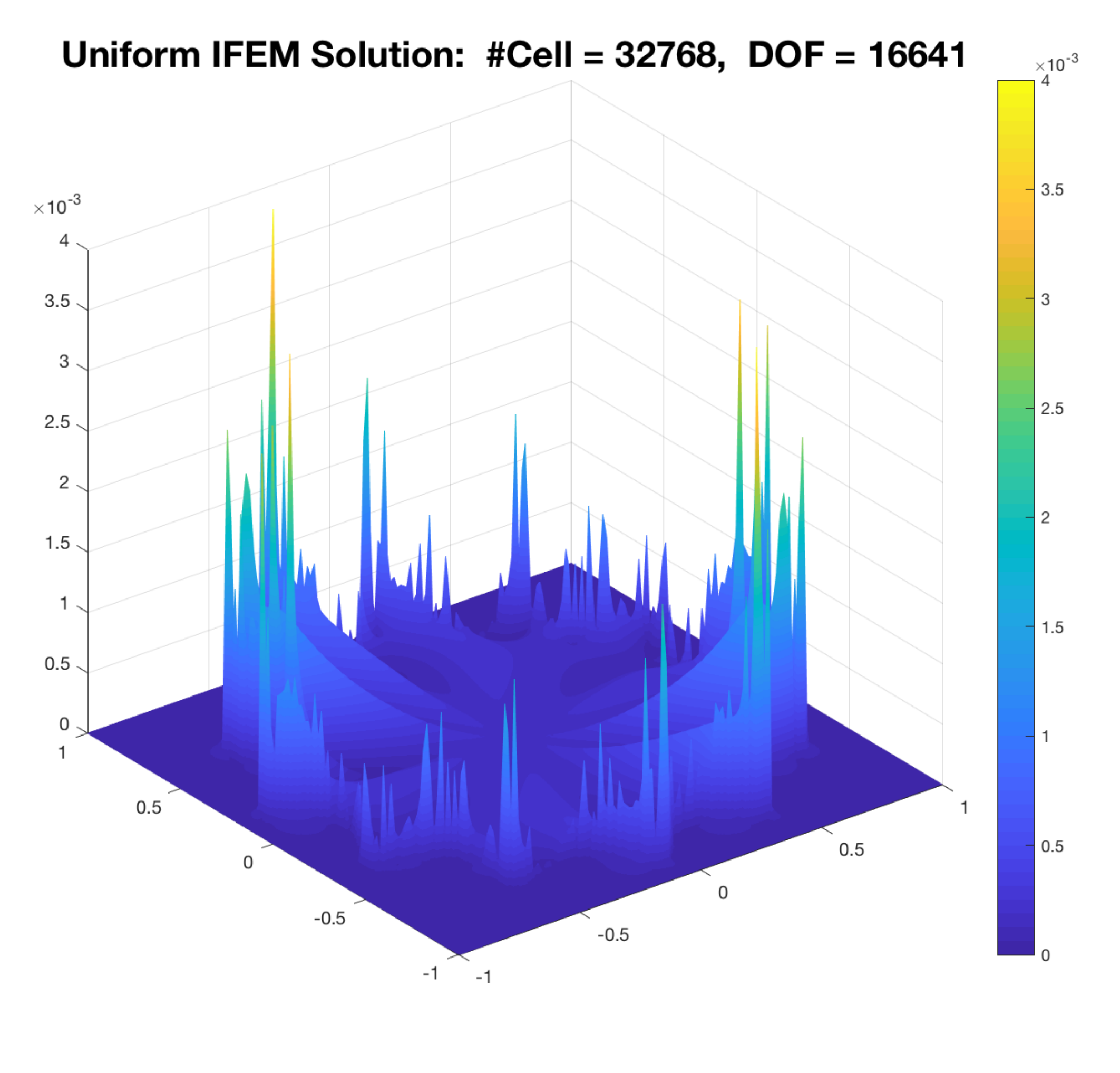}~
\includegraphics[width=0.48\textwidth]{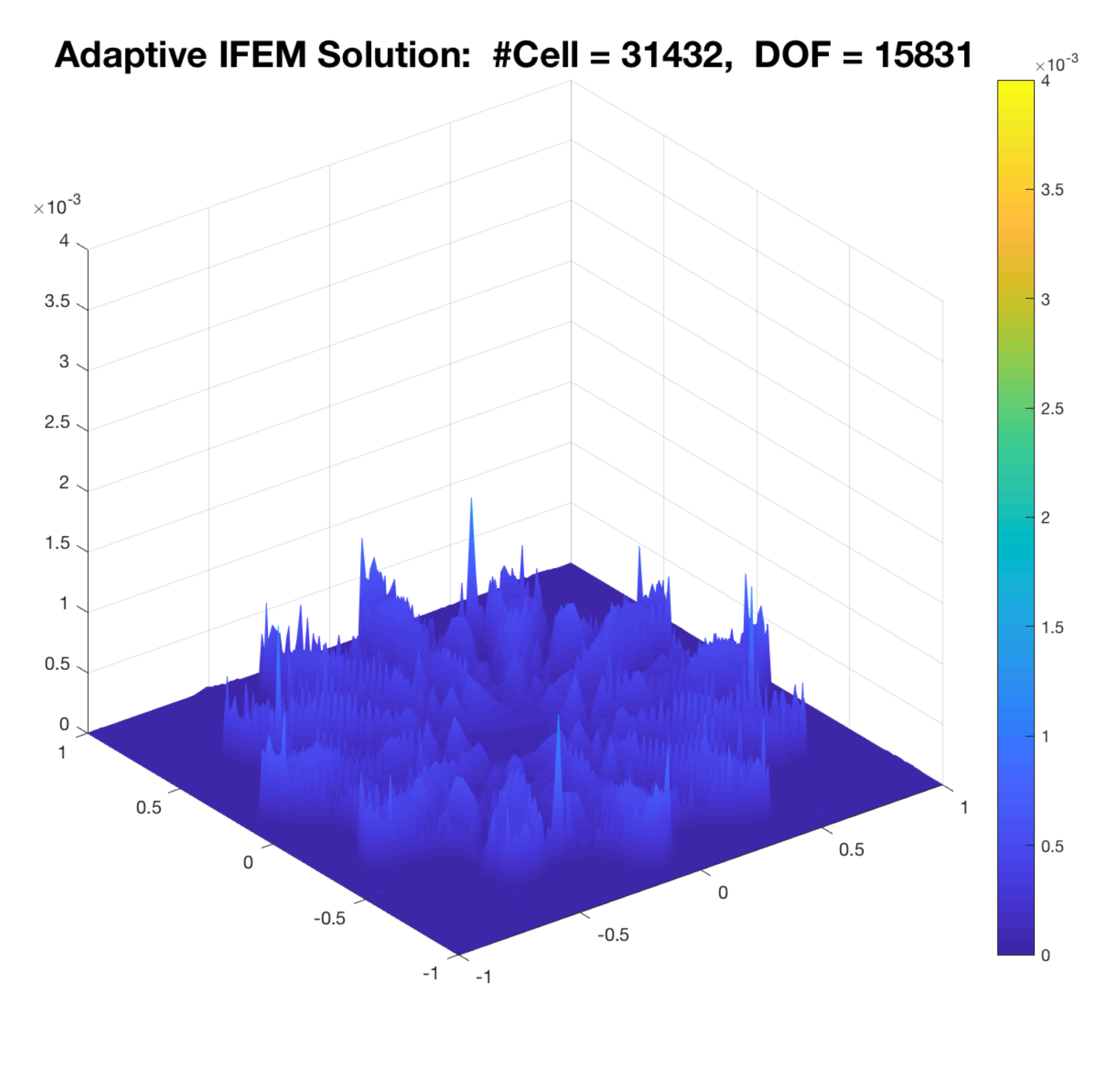}
\end{center}
\caption{Point-wise errors of uniform (left) and adaptive (right)  IFEM for Example \ref{ex4} with similar DOFs for Example 6.4}
\label{fig: err petal}
\end{figure}

\subsection*{Example 6.5 (Effect of the inconsistent error terms)}
\label{ex5}
In this example, we will explore the effect of inconsistent error term for different interface geometries by revisiting the Example 6.1 and Example 6.4 which has a simple ellipse interface and a more complicated petal interface. We consider the following error indicator 
\begin{equation} \label{xi-K}
\begin{split}
	\xi_K^2&=
		\sum \limits_{F \in \cE_K \cap \Ei} \left(\dfrac{h_F}{2} \|\tilde\a_F^{-1/2} j_{n,F}\|_{0,F}^2
			+\dfrac{h_F}{2}  \| \tilde\a_F^{1/2} j_{t,F}\|_{0,F}^2\right) 
						\\
					&+
			\sum \limits_{F \in \cE_K \cap \cE_I  \setminus  \Ei }
			  \dfrac{h_F}{2} \| \tilde\a_F^{-1/2}j_{n,F}\|_{0,F}^2
			  +
		\sum \limits_{F \in \cE_K \cap \cE_N}  h_F \| \tilde\a_F^{-1/2}j_{n,F}\|_{0,F}^2
\end{split}
\end{equation}
which is same as $\eta_K^2$ in \eqref{eta-K} but without the inconsistent error term $\|\tilde \a^{1/2} \nabla u_\cT \|_{S_K}^2$. The global error estimator is defined in the standard way: 
\[
\xi=\left( \sum_{K \in \cT} \xi_K^2\right)^{1/2}.
\]

First we compare the convergences using the error estimator $\xi$ and $\eta$ for Example 6.1. In the left plot of Figure \ref{fig: effect of Sk}, the estimators $\xi$ and $\eta$ are very similar at all meshes, and the errors of the corresponding IFEM solutions guided by these two estimators are also close. We believe that due to the simple and smooth interface shape of this example (an ellipse interface), the inconsistent error term is negligible. Next, we use the new error estimator $\xi$ in Example 6.4 In which the geometry of interface is more complicated. As we can see in the right plot of Figure  \ref{fig: effect of Sk}, the error estimators $\eta$ and $\xi$ and the corresponding IFEM solutions show notable differences, especially on the first few coarse meshes.  In this sense, including the geometrical correction (inconsistent error) term leads to better error indication particularly on coarse meshes. We also note that as the mesh is adaptively refined, the error estimators $\eta$ and $\xi$ become closer, as well as the IFEM solutions leading by these two estimators.

\begin{figure}[ht!]
\begin{center}
\includegraphics[width=0.49\textwidth]{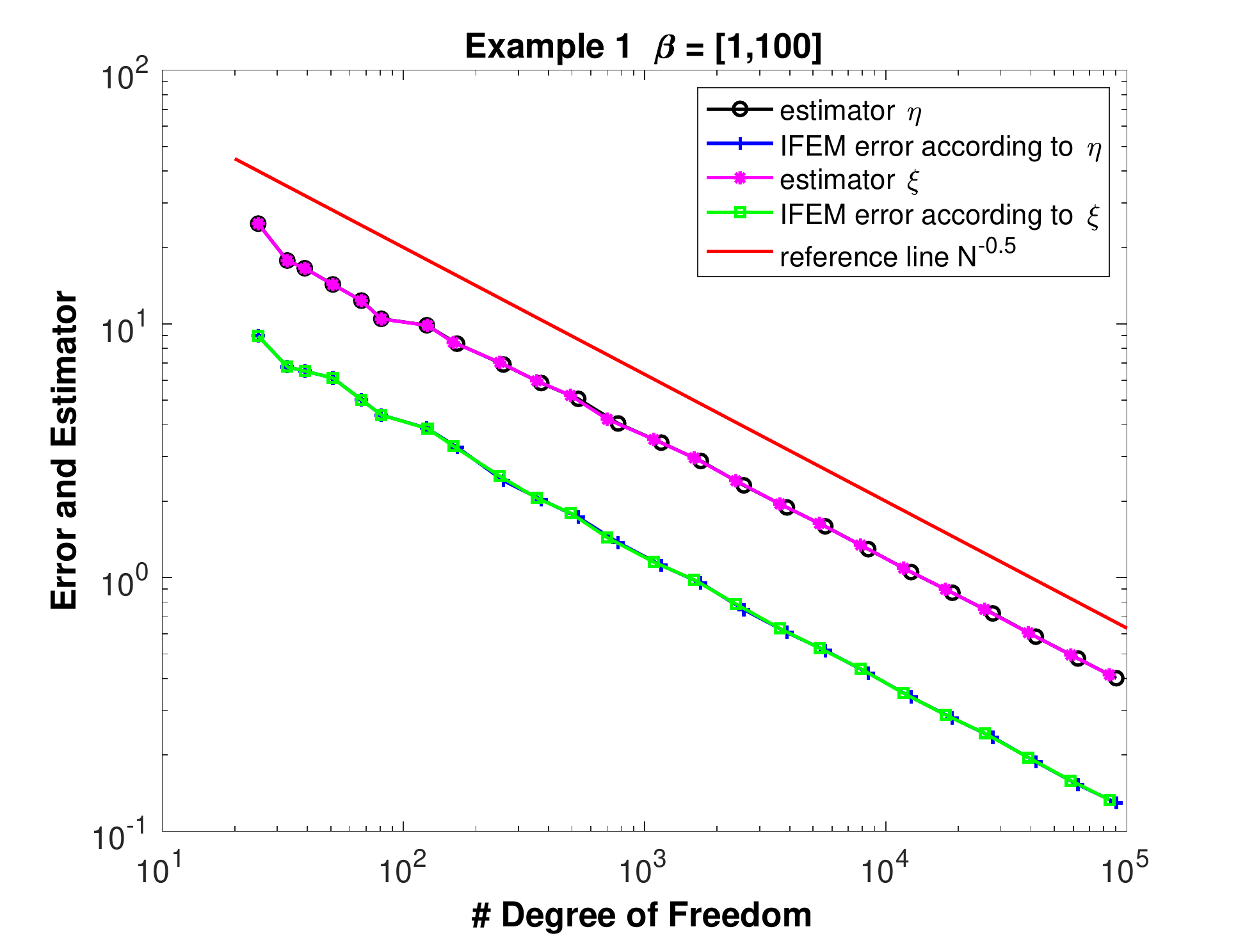}
\includegraphics[width=0.49\textwidth]{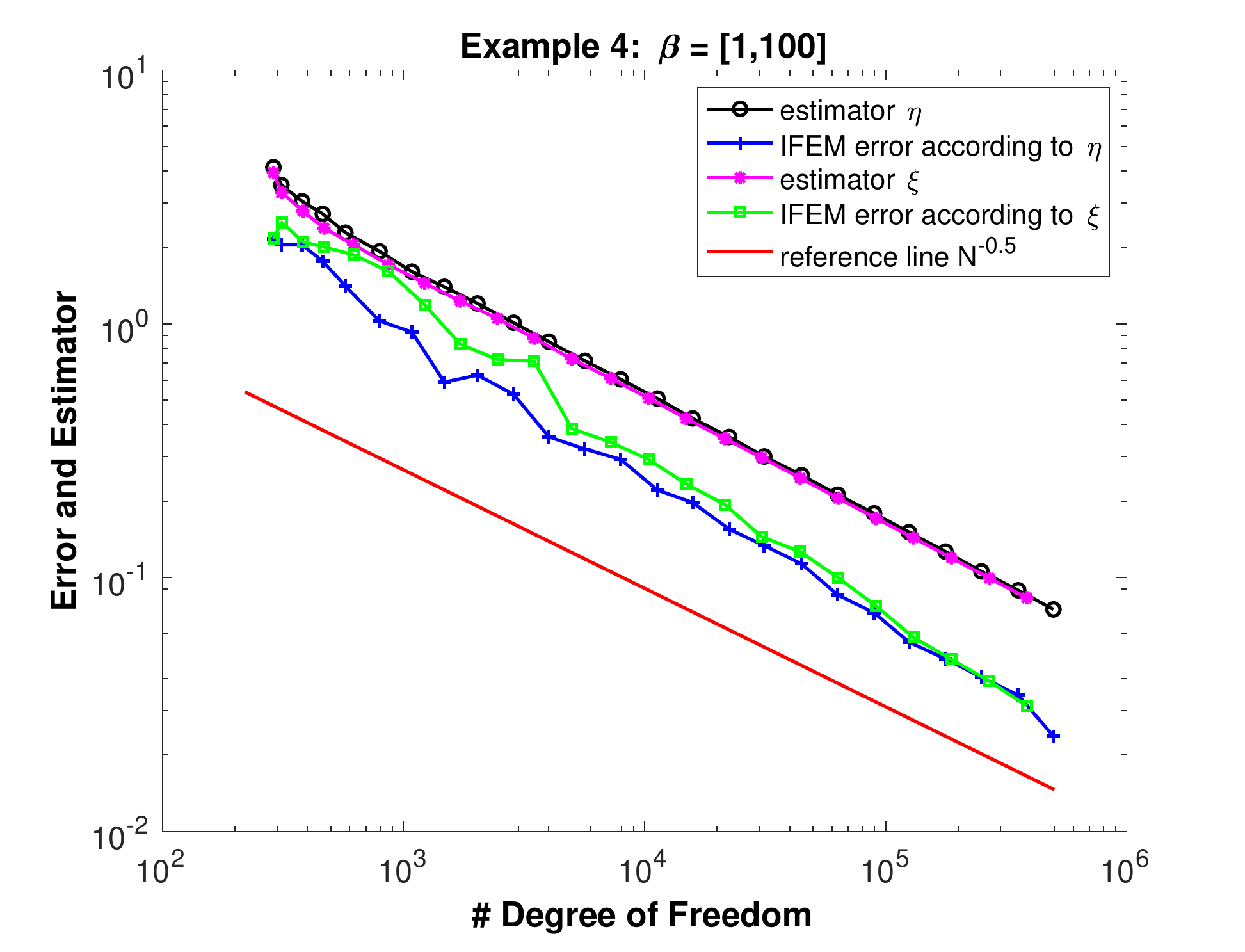}
\end{center}
\caption{Comparison of estimators $\eta$ and $\xi$ for Example 1 (left) and Example 4 (right)}
\label{fig: effect of Sk}
\end{figure}

\subsection*{Example 6.6  (Additional Comments on Large Jump Scenarios)}
\label{ex6}
In this test, we revisit the large jump scenario using test problem from the Example 6.2. To show that necessity of performing adaptive mesh strategy for IFEM, we use the true error in energy norm as our error indicator, i.e.,
\[\eta^*_K = \|\alpha^{1/2}(\nabla u-\nabla_h u_\cT)\|_{0,K},\]
and the global error estimator is the true error in the energy norm, i.e., 
\[
\eta^*=\left( \sum_{K \in \cT} \eta_K^{*2}\right)^{1/2} = \|\alpha^{1/2}(\nabla u-\nabla_h u_\cT)\|_{0,\O}.
\]
An adaptive mesh using the error estimator $\eta^*$ with similar number of triangles are shown in the left plot of Figure \ref{fig: estimator eta-star}. In both cases the refinement is concentrated around the interface, which is similar to the adaptive mesh in Figure \ref{fig: IFEM large}. This again shows that the partially penalized IFEM itself may not be sufficient to obtain accurate solution for extremely large jumps of  the coefficient. In this case, the adaptive mesh refinement is more advantageous. The right plot of Figure \ref{fig: estimator eta-star} shows the convergence of the errors governed by the estimator $\eta$ and the true error $\eta^*$. We can see that both converge in an optimal rate, although using the true error as estimator gives slightly more accurate solutions.

\begin{figure}[ht!]
\begin{center}
\includegraphics[width=0.43\textwidth]{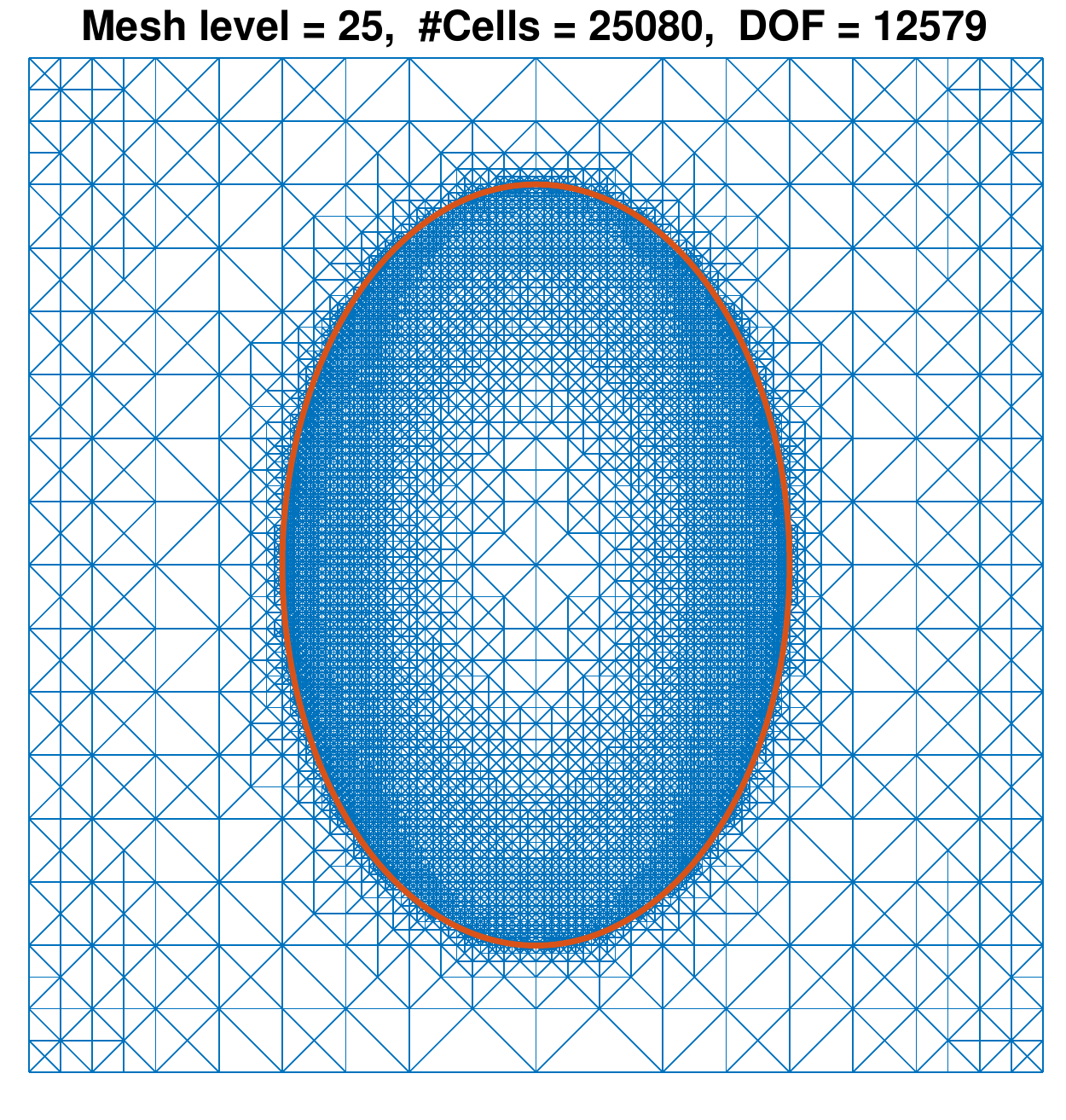}
\includegraphics[width=0.55\textwidth]{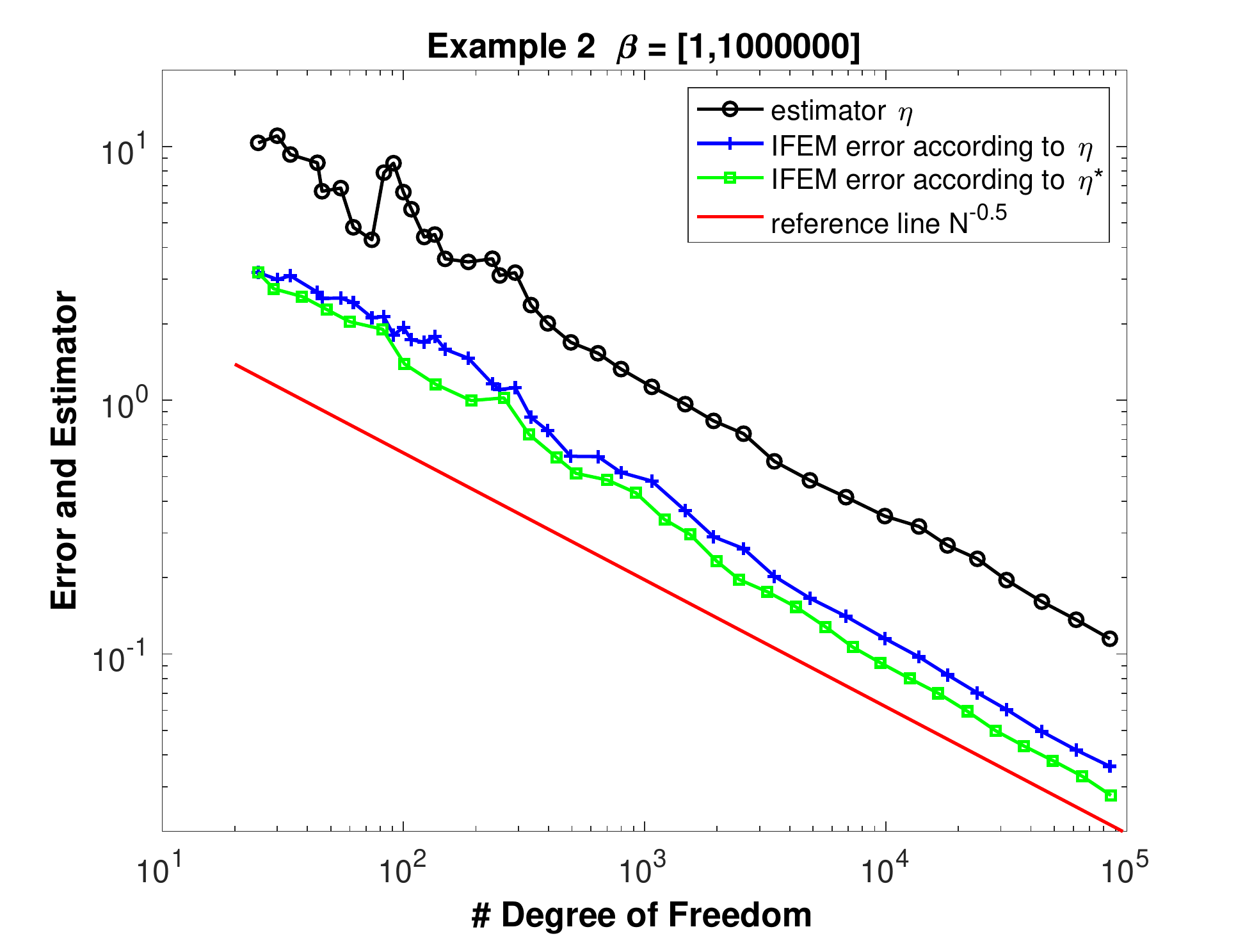}
\end{center}
\caption{Mesh generated by the adaptive IFEM (left) using true error as the error indicators
and convergence of IFEM solutions guided by $\eta$ and by the true error $\eta^*$ for Example 6.2}
\label{fig: estimator eta-star}
\end{figure}

 \providecommand{\bysame}{\leavevmode\hbox to3em{\hrulefill}\thinspace}
\providecommand{\MR}{\relax\ifhmode\unskip\space\fi MR }
\providecommand{\MRhref}[2]{%
  \href{http://www.ams.org/mathscinet-getitem?mr=#1}{#2}
}
\providecommand{\href}[2]{#2}


\begin{thebibliography}{10}
\bibitem{AdjeridGuoLin:2017}
{\sc S.Adjerid, R. Guo and T. Lin},
{\em High degree immersed finite element spaces by a least squares method}, 
Int. J. Numer. Anal. Model., {14} (2017), 604--626.

\bibitem{AdjeridLin:2009}
{\sc S.Adjerid and T. Lin},
{\em A {$p$}-th degree immersed finite element for boundary value problems with discontinuous coefficients}, 
Appl. Numer. Math., {59} (2009), 1303--1321. 
  
\bibitem{AinsworthRankin:08}
{\sc M. Ainsworth and R. Rankin},
{\em Fully computable bounds for the error in nonconforming finite element approximations of arbitrary order on triangular elements},
 SIAM J. Numer. Anal., 46 (2008), 3207--3232.
  
\bibitem{ArBrCoMa:02}
{\sc D. Arnold, F. Brezzi, B. Cockburn, and L. Marini},
{\em Unified analysis of discontinuous Galerkin methods for elliptic problems}, 
SIAM J. Numer. Anal., 39 (2002), 1749--1779.

\bibitem{BaOs:83}
{\sc I.~Babu{\v{s}}ka and J.~E. Osborn},
{\em Generalized finite element methods: their performance and their relation to mixed methods}, 
SIAM J. Numer. Anal., 20 (1983), 510--536. 

\bibitem{BeVe:00}
 {\sc C. Bernardi and R. Verf\"urth},
 {\em Adaptive finite element methods for elliptic equations with non-smooth coefficients},
 Numer. Math., 85 (2000), 579--608.
  
\bibitem{2007Braess}
{\sc D. Braess},
{\em Finite elements. Theory, fast solvers, and applications in elasticity theory},
Third edition, Cambridge University Press, Cambridge, 2007.


\bibitem{CaHeZh-mcom:17}
 {\sc Z. Cai, C. He and S. Zhang},
 {\em Residual-based a posteriori error estimate for interface problems: nonconforming linear elements},
Math. Comp., 86 (2017), 617--636.

\bibitem{CaHeZh-sinum:17}
{\sc Z. Cai, C. He and S. Zhang},
{\em Discontinuous finite element methods for interface problems: robust a priori and a posteriori error estimates},
SIAM J. Numer. Anal., 55 (2017), 400--418.

\bibitem{CaZhZh:17}
{ \sc W. Cao, X. Zhang, and Z. Zhang},
 \emph{Superconvergence of immersed finite element methods for interface problems},
  Adv. Comput. Math., 43 (2017), 795--821.

\bibitem{2017CaoZhangZhangZou}
W.~Cao, X.~Zhang, Z.~Zhang, and Q.~Zou.
\newblock Superconvergence of immersed finite volume methods for one-dimensional
  interface problems.
\newblock {\em J. Sci. Comput.}, 73 (2017), 543--565.
  
\bibitem{CaVe:99}
 {\sc C. Carstensen and R. Verf\"urth},
 {\em Edge residuals dominate a posteriori error estimates for low order finite element methods},
 SIAM J. Numer. Anal., 36 (1999), 1571--1587.

\bibitem{ChXiZh:09}
{\sc Z. Chen, Y. Xiao, and L. Zhang},
 \emph{The adaptive immersed interface finite element method for elliptic and Maxwell interface problems}, 
J. Comput. Phys., {228} (2009),  5000--5019.

\bibitem{ChZou:98}
{\sc Z. Chen and J. Zou},
\emph{Finite element methods and their convergence for elliptic and parabolic interface problems},
Numer. Math., {79} (1998), 175--202. 


\bibitem{DaDuPaVa:1996}
{\sc E. Dari, R. Duran, C. Padra, and V. Vampa},
{\em A posteriori error estimators for nonconforming finite element methods},
RAIRO Mod\'el Anal. Num\'er., 30 (1996), 385--400.

\bibitem{GuSaSa:16}
{\sc J. Guzm{\'a}n, M. S{\'a}nchez, and M. Sarkis},
 \emph{On the accuracy of finite element approximations to a class of interface problems},
Math. Comp., {85} (2016), 2071--2098. 

\bibitem{HaHa:02}
{\sc A. Hansbo and P. Hansbo},
\emph{An unfitted finite element method, based on {N}itsche's method, for elliptic interface problems}, 
Comput. Methods Appl. Mech. Engrg., {191} (2002), 5537--5552. 

\bibitem{2011HeLinLin}
{\sc X. He, T. Lin, and Y. Lin},
 \emph{Immersed finite element methods for elliptic interface problems with non-homogeneous jump conditions}, 
Int. J. Numer. Anal. Model., {2} (2011), 284--301.

\bibitem{HouWuCai:99}
{\sc T. Hou, X. Wu, and Z. Cai},
 \emph{Convergence of a multiscale finite element method for elliptic problems with rapidly oscillating coefficients}, 
Math. Comp., {68} (1999), 913--943.

\bibitem{LeLi:94}
{\sc R. LeVeque and Z. Li},
 \emph{The immersed interface method for elliptic equations with discontinuous coefficients and singular sources},
  SIAM J. Numer. Anal., {31} (1994), 1019--1044.
  
\bibitem{LiLinLinRo:04}
{\sc Z. Li, T. Lin, Y. Lin, and R. Rogers},
\emph{An immersed finite element space and its approximation capability}, 
Numer. Methods Partial Differential Equations, {20} (2004), 338--367. 

\bibitem{Li:98}
{\sc Z. Li},
\emph{The immersed interface method using a finite element formulation},
Appl. Numer. Math., {27} (1998), 253--267.

\bibitem{2006LiIto}
{\sc Z. Li and K. Ito}, 
\emph{The immersed interface method}, Frontiers in Applied Mathematics, vol.~33, 
  Society for Industrial and Applied Mathematics (SIAM), Philadelphia, PA, 2006.
  
\bibitem{2003LiLinWu}
{\sc Z. Li, T. Lin, and X. Wu},
\emph{New {C}artesian grid methods for interface problems using the finite element formulation}, 
Numer. Math., {96} (2003), 61--98. 

\bibitem{LinLinZhang:15}
{\sc T. Lin, Y. Lin, and X. Zhang},
\emph{Partially penalized immersed finite element methods for elliptic interface problems}, 
SIAM J. Numer. Anal., {53} (2015), 1121--1144. 

\bibitem{2018LinSheenZhang}
{\sc T. Lin, D. Sheen, and X.Zhang},
\emph{A nonconforming immersed finite element method for elliptic interface problems}, 
J. Sci. Comput., in press (2018).  


\bibitem{LinYangZhang:15:2}
{\sc T. Lin, Q.Yang, and X.Zhang},
\emph{{\it A Priori} error estimates for some discontinuous {Galerkin} immersed finite element methods}, 
J. Sci. Comput., {65} (2015), 875--894.

\bibitem{Maubach:95}
{\sc J. Maubach},
 \emph{Local bisection refinement for {$n$}-simplicial grids generated by reflection}, 
SIAM J. Sci. Comput., {16} (1995), 210--227. 

\bibitem{1999MoesDolbowBelytschko}
{\sc N. Mo\"{e}s, J. Dolbow, and T. Belytschko},
 \emph{A finite element method for crack growth without remeshing}, 
 Internat. J. Numer. Methods Engrg., {46} (1999), 131--150. 
 
\bibitem{2016MuWangYeZhao}
{\sc L. Mu, J. Wang, X. Ye, and S. Zhao},
 \emph{A new weak {G}alerkin finite element method for elliptic interface problems}, 
J. Comput. Phys., {325} (2016), 157--173. 

\bibitem{2002Peskin}
{\sc C. Peskin},
\emph{The immersed boundary method}, 
Acta Numer., {11} (2002), 479--517. 

\bibitem{2010VallaghePapadopoulo}
{\sc S. Vallagh{\'e} and T. Papadopoulo}, 
\emph{A trilinear immersed finite element method for solving the electroencephalography forward problem}, 
SIAM J. Sci. Comput., {32} (2010), 2379--2394.

\bibitem{Ve:1996}
{\sc R. Verf\"urth},
{\em A Posteriori Error Estimation Techniques for Finite Element Methods},
Oxford University Press, Oxford, United Kingdom, 2013.

\bibitem{2011WuLiLai}
{\sc C. Wu, Z. Li, and M. Lai},
\emph{Adaptive mesh refinement for elliptic interface problems using the non-conforming immersed finite element method},
Int. J. Numer. Anal. Model., {8} (2011), 466--483.

\bibitem{2004ZhaoWei}
{\sc S. Zhao and G. Wei},
\emph{High-order {FDTD} methods via derivative matching for {M}axwell's equations with material interfaces}, 
J. Comput. Phys., {200} (2004), 60--103. 

\end{thebibliography}
\end{document}